\documentclass[AMS,STIX1COL,SYSTEMFONTS,doublespace]{WileyNJD-v2}

\usepackage[shortlabels]{enumitem}
\newlist{algolist}{enumerate}{3}
\setlist[algolist,1]{label=\arabic*., ref=\arabic*.}
\setlist[algolist,2]{label=(\alph*), ref=\arabic{algolisti}-(\alph*)}
\setlist[algolist,3]
{label=\roman*., ref=\arabic{algolisti}-(\alph{algolistii})-\roman*}

\usepackage[justification=centering]{caption}
\usepackage[labelformat=simple]{subcaption}
\usepackage{mathtools}            
\usepackage[capitalize,noabbrev]{cleveref}   
\makeatletter
\if@cref@capitalise
\crefname{algolisti}{Step}{Steps}
\crefname{algolistii}{Step}{Steps}
\crefname{algolistiii}{Step}{Steps}
\else
\crefname{algolisti}{step}{steps}
\crefname{algolistii}{step}{steps}
\crefname{algolistiii}{step}{steps}
\fi
\makeatother

\DeclarePairedDelimiter\abs{\lvert}{\rvert} 
\DeclareMathOperator*{\arginf}{arg\,inf} 
\DeclareMathOperator*{\argmax}{arg\,max} 
\DeclareMathOperator*{\argmin}{arg\,min} 
\DeclareMathOperator*{\argsup}{arg\,sup} 
\providecommand{\CC}{{{C\nolinebreak[4]\hspace{-.05em}\raisebox{.4ex}{\tiny\bf 
++}}}}
\DeclarePairedDelimiter\ceil{\lceil}{\rceil} 
\providecommand{\N}{\mathbb{N}} 
\providecommand{\R}{\mathbb{R}} 
\providecommand{\Z}{\mathbb{Z}} 
\DeclarePairedDelimiter\set{\{}{\}}
\DeclarePairedDelimiterX\setc[2]{\{}{\}}{\,#1 \;\delimsize\vert\; #2\,}
\providecommand{\vc}[1]{\mathbf{#1}}

\newcommand{\B}{Bertsekas}
\newcommand{\BC}{Bertsekas and Casta\~{n}\'{o}n}

\articletype{Article Type}%

\received{XX June 2018}
\revised{XX YY 201Z}
\accepted{XX YY 201Z}

\begin{document}

\title{A real-valued auction algorithm for optimal transport\protect\thanks{This 
material is based upon work supported by the National Science Foundation 
Graduate Research Fellowship under Grant No.\ DGE--1148903.}}

\author[1]{J.D.\ Walsh III*}

\author[2]{Luca Dieci}

\authormark{Real-valued auction algorithm for optimal transport}

\address[1]{\orgname{Naval Surface Warfare Center, Panama City Division}, \orgaddress{\state{Florida}, \country{United States}}}

\address[2]{\orgdiv{School of Mathematics}, \orgname{Georgia Institute of Technology}, \orgaddress{\state{Georgia}, \country{United States}}}

\corres{*J.D.\ Walsh III, NSWC PCD (X24), 110 Vernon Ave, Panama City, FL 32407 USA. \email{joseph.d.walsh@navy.mil}}


\abstract[Summary]{Optimal transportation theory is an area of mathematics with real-world applications in fields ranging from economics
to optimal control to machine learning. We propose a new algorithm for solving discrete transport (network flow) problems,
based on classical auction methods.
Auction methods were originally developed as an alternative to the Hungarian method for the assignment problem, so the classic
auction-based algorithms solve integer-valued optimal transport by converting such problems into assignment problems.
The general transport auction method we propose works directly on real-valued transport problems.
Our results prove termination, bound the transport error, and relate our algorithm to the classic algorithms of Bertsekas and Casta\~{n}\'{o}n.}


\jnlcitation{\cname{%
\author{Walsh J.}, and 
\author{L. Dieci}} (\cyear{2018}), 
\ctitle{General auction method for real-valued optimal transport}, \cjournal{Statistical Analysis and Data Mining}, \cvol{CoDA 2018 special issue}.}


\maketitle

\section{Introduction}
Numerical optimal transport is currently a significant and
exciting area of research.
The ability to quickly approximate discrete and semi-discrete transport problems
has important real-world applications.
Even when the original problem is continuous, existing solution methods
solve by considering a discrete or semi-discrete approximation.

At first glance, auction algorithms appear to be a good approach to solving
such discrete and semi-discrete transport problems. Unlike many other network
methods, they offer controlled approximation and bounded error results.

Furthermore, auction methods work best when applied to problems characterized
by one or both of two properties: \emph{homogeneity} and
\emph{asymmetry}~\cite[p.\ 58--59]{Bertsekas1992a}.
Homogeneity means the sinks and sources have a small range of weights, while
asymmetric problems have far more sinks than sources (or sources than sinks).
This appears promising, as many semi-discrete transport problems are
characterized by strong asymmetry.

However, real-valued transport problems are not necessarily homogeneous, and
the transport costs they use are unlikely to be
\emph{rationally-related}.\footnote{A set $S \subset \R$ is \emph{rationally-related}
if there exists $t \in \R$ such that, for all $s \in S$, $st \in \Z$.}
In particular, costs are often derived by computing the value of a $p$-norm
for some $p \neq 1$.
This lack of rational-relatedness is crucial, since existing auction methods put
integer-based
restrictions on the cost and weight data inherent to the problems they can solve.

The complexity arguments used by existing auction methods require that the 
cost of each matching is an integer.
If rational (or rationally-related) costs are required, all costs must
be transformed into integers using a common denominator.
This can impact the worst-case time complexity of the algorithm.
If non-rationally-related costs are required, as in the $p$-norm instance
described above, these auction methods offer no
guarantee of termination or convergence.

When existing auction methods are used to solve real-valued network flow or optimal 
transport problems,
then another, more significant, data restriction arises.
Existing auction methods work by converting the problem into an assignment
problem, where all weights are equal.
Hence, as with the costs, if rationally-related weights are required, the entire set of 
weights must be transformed using a common denominator.
Here, though, all weights must be unit-valued, and the common
denominator directly impacts both the storage 
requirements and the worst-case time complexity of the resulting problem.
Furthermore, because the final weights must be unitary,
if non-rationally-related weights are required,
then these methods cannot be applied at all.

We propose a more general auction method, one developed specifically for the 
transport problem (rather than one based on the assignment problem) and capable 
of handling real-valued costs and weights.
To support this extension, we also prove convergence and provide \emph{a 
priori} error bounds for transport problems with real-valued data.
Finally, we compare auction methods using a series of standard problems and 
show results indicating complexity roughly comparable to that of the original 
auction method, in cases where the latter can be applied.

\section{Background}
The term ``optimal transport'' is generally used to refer to a continuous form
known as the Monge-Kantorovich problem. One way to define the problem, derived
from~\cite[p.\ 2]{Villani2003a}, follows:
\begin{definition}[Monge-Kantorovich problem]\label{MKproblem}
Let $X$ and $Y$ be Polish spaces, let $\mu$ and $\nu$ be probability densities 
defined on $X$ and $Y$, and let $c(\vc{x},\,\vc{y}) : X \times Y \to \R$ be a 
lower semi-continuous
\emph{ground cost} function.
Define the set of \emph{transport plans}
\begin{equation}\label{MK-1}
\Pi(\mu,\,\nu) := \set*{\pi \in \mathcal{P}(X \times Y) \left|
\begin{array}{c}
\pi[A \times Y] = \mu[A],\,
 \pi[X \times B] = \nu[B] \ ,\\
 \forall \text{ meas.\ }
 A \subseteq X,\, B \subseteq Y
\end{array}
\right. },
\end{equation}
where $\mathcal{P}(X \times Y)$ is the set of probability measures on the 
product space,
and define the \emph{primal cost} function $P: \Pi(\mu,\,\nu) \to \R$ as
\begin{equation}\label{primal_cost}
P(\pi) := \int_{X \times Y} c(\vc{x},\,\vc{y})\, d\pi(\vc{x},\,\vc{y}).
\end{equation}
The Monge-Kantorovich problem is to
find the \emph{optimal primal cost}
\begin{equation}\label{MK-2}
P^* := \inf_{\pi \in \Pi(\mu,\,\nu)}\, P(\pi),
\end{equation}
and an associated \emph{optimal transport plan}
\begin{equation}\label{MK-pi-star}
\pi^* := \arginf_{\pi \in \Pi(\mu,\,\nu)}\, P(\pi).
\end{equation}
\end{definition}
\cref{MKproblem} is easily generalized. 
For larger measure spaces, assume the measures $\mu$ and $\nu$ on
$X$ and $Y$ are both multiplied by some constant $L$.
To find an optimal maximum cost, simply negate $c$.
The definition can also be applied to semi-continuous or discrete
problems, discretizing one or both measure spaces and replacing
the corresponding integral(s) with a finite sum.

Monge's original formulation assumed masses could not be split.
The relaxed form above, developed by Kanorovich, allows
splitting masses, and is typically preferred because it guarantees
the existence of a solution if the requirements above are satisfied. 
Kantorovich also identified the optimal transport problem's dual formulation.
\begin{definition}[Dual formulation]\label{KantoDual}
Define the set of functions
\begin{equation}
\Phi_c(\mu,\,\nu) := \set*{(\varphi,\,\psi) \in L^1(d\mu) \times L^1(d\nu)
\left|
\begin{array}{c}
\varphi(\vc{x}) + \psi(\vc{y}) \leq c(\vc{x},\,\vc{y}) \ ,\\
d\mu \text{ a.e.\ }
\vc{x} \in X,\,d\nu\text{ a.e.\ } \vc{y} \in Y
\end{array}
\right. }.
\end{equation}
Let the \emph{dual cost} function,
$D: \Phi_c(\mu,\,\nu) \to \R$, be defined as
\begin{equation}\label{e:dualCost}
D(\varphi,\,\psi) := \int_X \varphi \,d\mu + \int_Y \psi \,d\nu.
\end{equation}
Then, the \emph{optimal dual cost} is
\begin{equation}
D^* := \sup_{(\varphi,\,\psi) \in \Phi_c(\mu,\,\nu)}\, D(\varphi,\,\psi),
\end{equation}
and an optimal dual pair is given by
\begin{equation}
(\varphi^{*},\,\psi^{*}) := \argsup_{(\varphi,\,\psi) \in \Phi_c(\mu,\,\nu)}\, 
D(\varphi,\,\psi).
\end{equation}
\end{definition}
The continuous and semi-continuous Monge-Kantorovich problems were primarily
of interest to analysts until the
Polar Factorization theorem of~\cite{Brenier1991a} revealed deep connections
between optimal transportation and partial differential equations.
This discovery paved the way for major improvements in computational optimal
transport, such as the BFO finite-difference solver described in~\cite{Benamou2014a}
and the iterative Bregman projection method proposed in~\cite{Cuturi2013a}.
For a more detailed overview of these developments, see~\cite{Peyre2019a}.

These new approaches come with new limitations, needed to satisfy the
well-posedness requirements of partial differential equation methods.
These limitations are formally expressed by the Ma-Trudinger-Wang conditions
described in~\cite{Ma2005a}. In practice the MTW conditions are satisfied by the
use of strictly convex ground cost functions, almost exclusively limited to the
squared Euclidean distance.

For this reason, there remains a great deal of interest in discrete optimal transportation
methods. All optimal transport solvers require some form of discretization, so the
trade-off can generally be seen as exchanging some generality of measure for
certain assumptions of continuity. For one approach to this trade-off, which combines
discrete methods with limited continuity assumptions to solve
semi-discrete transport problems with $p$-norm ground costs, see~\cite{Dieci2019a}.

The study of computational discrete optimal transportation has its origins in linear
programming, particularly the simplex method developed by George Dantzig
in the 1940s; for example, see~\cite{Dantzig1951a}. Prior to the publication of
the auction method, most discrete transportation algorithms had their origins
in linear programming.

The auction method, first proposed by Dimitri \B{} in the late
1970s~\cite{Bertsekas1981a}, uses a relaxation approach derived from techniques
for solving partial differential equations.
\B{} developed the method for the assignment problem, as an alternative to the 
Hungarian method~\cite{Kuhn1955a}.
In 1989 \BC{} extended the original auction method to solve minimal cost 
flow and optimal transport problems by taking into account what they call 
``similar persons and objects''~\cite{Bertsekas1989a}.
Their extended auction method decomposes the transport problem into an 
equivalent assignment problem by splitting each supply and demand vertex into 
multiple identical vertices of unitary weight. The number of copies is equal to 
the supply or demand ``weight.'' The extended auction method then solves the 
assignment problem and combines the resulting assignments to provide the desired 
transport solution.
When all costs and weights are integers, these two auction methods offer
worst-case error and complexity bounds, as well 
as conditions under which an optimal solution is guaranteed.
The auction method's potential application to generalized
optimal transport problems is already known to the wider computational
optimal transport community;
e.g., the discussion in~\cite{Merigot2013a}.

With these factors in mind, now consider
the discrete real-valued optimal transport problem itself.
We will use this common frameworks to unify notation and interrelate key concepts
as we briefly describe the various auction methods.
When algorithmic restrictions limit us to a special-case of the transport 
problem, we describe the required conditions.

\subsection{Transport problem}
\label{sn:trans_prob}
Consider the transport problem $\mathcal{T}$, which we define using
linear programming:
\begin{definition}[Discrete Optimal Transport]\label{DiscTransport}
Suppose we are given a \emph{demand} vector $\set{d_i}_{i=1}^M$ and a
\emph{supply} vector $\set{s_j}_{j=1}^N$, whose demand coefficients $d_i$ and 
supply coefficients $s_j$ are positive scalars such that
\begin{equation}\label{eq:tp_defL}
L := \sum_{i=1}^M d_i = \sum_{j=1}^N s_j > 0.
\end{equation}
We refer to $L$ as the \emph{total weight} of the transport problem.
In the underlying transport graph, the vertex $i$ associated with demand 
coefficient $d_i$ is a \emph{sink}, and the vertex $j$ associated with supply 
coefficient $s_j$ is a \emph{source}.
We denote the set of sinks by $I$, and the set of sources by $J$.

Furthermore, suppose for each sink $i$ we are given the nonempty set $A(i)$ of 
sources to which the sink $i$ is adjacent.
The set of all possible transport pairs is equal to
\begin{equation}
\mathcal{A} := \setc*{(i,\,j)}{j \in A(i),\, i \in \set{1,\,\ldots,\,M}\,}.
\end{equation}
Thus, $\mathcal{A}$ is the set of arcs of the underlying transport graph, a 
bipartite graph with $M+N$ vertices and $\abs{\mathcal{A}} \leq MN$ arcs.

For each $(i,\,j) \in \mathcal{A}$, let $c_{ij} < 0$ be the given \emph{cost 
coefficient} (or simply \emph{cost}).
Our goal is to
\begin{subequations}
\begin{alignat}{2}
\renewcommand{\arraystretch}{1.8}
\text{ maximize }
&\quad
\sum\limits_{(i,\,j) \in \mathcal{A}} c_{ij}f_{ij}
& &
\label{eq:transMAX}
\\
\text{ subject to }
&
\sum\limits_{\setc{j}{j \in A(i)}} f_{ij} = d_i
& &
\forall\, i \in \set{ 1,\,\ldots,\,M },
\label{eq:transDi}
\\
&
\sum\limits_{\setc{i}{j \in A(i)}} f_{ij} = s_j
& &
\forall\, j \in \set{ 1,\,\ldots,\,N },\text{ and }
\label{eq:transSj}
\\
&
0 \leq f_{ij} \leq \min\set{d_i,\,s_j}
& &
\forall\, (i,\,j) \in \mathcal{A}.\label{eq:transQij}
\end{alignat}
\end{subequations}
We refer to $f_{ij}$ as the \emph{flow} along $(i,\,j)$, because it gives the 
amount transported (i.e.\ ``flowing'') from source $j$ to sink $i$.
\end{definition}
We have assumed negative costs and formulated the transport problem 
as a maximization problem,
in accordance with the standard
implementation of the auction method.
Because $c_{ij} < 0$, the maximization equation \cref{eq:transMAX} provides a 
minimum overall cost (obtained by reversing the sign on each $c_{ij}$). We 
refer to
\begin{equation}
\sum\limits_{(i,\,j) \in \mathcal{A}} c_{ij}f_{ij}
\end{equation}
as the \emph{primal cost}.
The solution to \cref{eq:transMAX} is called the \emph{optimal primal 
cost}, or \emph{optimal cost}, of the transport problem, and is denoted by 
$P^*$. If at least one set of flows $\set{ f_{ij} }$ exists such that
\cref{eq:transDi,eq:transSj,eq:transQij} are satisfied,
the optimal cost $P^*$ exists and is unique.

In contrast, a set of flows $\set{ f_{ij}^* }$ that achieves $P^*$ is almost never
unique. See~\cite{Cuesta1993a} for a condition that is sufficient to ensure
the uniqueness of the set of flows\footnote{The condition guarantees uniqueness a.e., but
for discrete measures uniqueness a.e.\ is equivalent to actual uniqueness}.
However, a set of flows $\set{ f_{ij}^* }$ that achieves $P^*$ is also bijective,
in the following sense:
suppose the direction of flow is reversed, which is equivalent to a transport problem
$\mathcal{T}'$ with demand vector $\set{s_i}_{i=1}^N$ and supply vector
$\set{d_j}_{j=1}^M$.
For all $i \in \N_N$ and $j \in \N_M$, define the flow $f_{ij}' = f_{ji}^*$.
Then $\set{ f_{ij}' }$ achieves the optimal transport cost for $\mathcal{T}'$,
and that optimal transport cost equals $P^*$.

\subsection{Transport plan}
A \emph{transport plan} (or \emph{transport map}) $T$ is a multiset of triples
$(i,\,j;\,q_{ij})$ such that $(i,\,j) \in \mathcal{A}$ and
the transported \emph{quantity}, $q_{ij}$, is non-negative.
Note that $T$ may be empty.
While the elements of $T$ are not necessarily unique, for each $(i,\,j) \in 
\mathcal{A}$ we can compute the unique flow $f_{ij}$ given by $T$ as
\begin{equation}
f_{ij} = \sum_{\setc{(k,\,l;\,q_{kl}) \in T}{(k,\,l) = (i,\,j)}} q_{kl}.
\end{equation}
In order to apply the plan to our transport problem, we require that $T$ 
satisfies $f_{ij} \leq \min\set{d_i,\,s_j}$ for all $(i,\,j) \in \mathcal{A}$.
By a minor abuse of notation, we may say $(i,\,j) \in T$ to signify that 
$(i,\,j;\,q_{ij}) \in T$ for some $q_{ij} > 0$.
We may also say $T \in \mathcal{T}$ to refer to some 
transport  plan $T$ associated with the transport problem $\mathcal{T}$.

Given any transport plan $T$, we say that sink $i$ is \emph{satisfied} if
\begin{equation}
\sum\limits_{\setc{q_{ij}}{(i,\,j,\,q_{ij}) \in T}} q_{ij} = d_i.
\end{equation}
Otherwise, we say that $i$ is \emph{unsatisfied}. (Alternatively, we may say 
$i$ has unsatisfied demand $D_i$, where $0 \leq D_i \leq d_i$.)

Similarly, when
\begin{equation}
\sum\limits_{\setc{q_{ij}}{(i,\,j,\,q_{ij}) \in T}} q_{ij} = s_j,
\end{equation}
we say the source $j$ is \emph{unavailable}.
Otherwise, we say that $j$ is \emph{available}, or that $j$ has available 
supply 
$S_j$, where $0 \leq S_j \leq s_j$.

A transport plan is said to be \emph{feasible}
or \emph{complete} when all sinks are satisfied; otherwise the plan is called
\emph{partial}\footnote{Linear programming
terminology used to indicate the status of~\cref{eq:transDi,eq:transSj,eq:transQij};
for more background, see ``Network Flow Problems,'' Part III of~\cite{Chvatal1983a}.}.

If $T$ is such that the pair $(i,\,j)$ appears at most once, and 
$(i,\,j;\,q_{ij}) \in T$ implies $q_{ij} > 0$, we refer to $T$ as a 
\emph{simplified} transport plan. In this case $q_{ij} = f_{ij}$, and we may 
refer to flow and quantity interchangeably.
Simplified transport plans, while not strictly necessary, greatly improve
clarity of notation.
We will generally assume $T$ is simplified when stating definitions and proofs.

\subsection{Dual problem}
We can write the dual transport problem as
\begin{equation}
\min_{u_i,\,p_j} \set*{ \sum_{i=1}^M d_iu_i + \sum_{j=1}^N s_jp_j },
\end{equation}
with the restriction that $u_i + p_j \leq c_{ij}$ for all $(i,\,j) \in 
\mathcal{A}$.
We call the dual variable $p_j$ a \emph{price} of $j$, and the vector
$p = \set{p_j}_{j=1}^N$ a price vector of $\mathcal{T}$.
Assume $p_j \geq 0$ for all $j$.

Given some price vector $p$, the \emph{expense} associated with the arc 
$(i,\,j) \in \mathcal{A}$ is
\begin{equation}
x_{ij} := c_{ij} - p_j
\end{equation}
and the \emph{expense} for the sink $i$ is
\begin{equation}\label{eq:expn}
x_i := \max_{j \in A(i)} x_{ij} = \max_{j \in A(i)} \set{ c_{ij} - p_j }.
\end{equation}
Because $c_{ij} < 0$, the maximization of the expense $x_i$ actually generates 
the least overall expense (obtained by reversing the signs of the $x_{i}$s).

Suppose we have a simplified complete transport plan $T$ and a price vector $p$.
From linear programming theory, we know $(T,\,p)$ is simultaneously primal and 
dual optimal if and only if
\begin{equation}
u_i = \max_{k \in A(i)} \set{ c_{ik} - p_k } = c_{ij} - p_j
\quad\quad
\forall\,
(i,\,j;\,f_{ij}) \in T.
\end{equation}
In other words, the expense for each sink is minimized by transport to the 
least expensive source(s).
This is known as the \emph{complementary slackness condition}, or 
\emph{complementary slackness}.

Among other things, complementary slackness implies that the dual problem can 
only be minimized when $u_i = x_i$ for all $i$.
Thus, we can view the prices $p_j$ as the only variables in our dual problem, 
which we can rewrite as
\begin{equation}\label{eq:dualMIN}
\min_{p=\set{p_j}_{j=1}^N}
\set*{ \sum_{i=1}^M d_i\max_{j \in A(i)} \set{c_{ij} - p_j}
+ \sum_{j=1}^N s_jp_j }.
\end{equation}
Because $c_{ij} < 0$, the minimization given in \cref{eq:dualMIN} provides a 
maximum overall profit (obtained by reversing the sign of each $c_{ij}$).
We refer to
\begin{equation}
\sum_{i=1}^M d_i\max_{j \in A(i)} \set{c_{ij} - p_j} + \sum_{j=1}^N s_jp_j
\end{equation}
as the \emph{dual profit}.
The solution to \cref{eq:dualMIN} is called the \emph{optimal dual 
profit}, or \emph{optimal profit}, of the transport problem, and is denoted by 
$D^*$.

\subsection{Auction fundamentals}
The auction method solves the dual problem described above, using a relaxation
technique inspired by the open ascending price process commonly used for 
real-world auctions.
Each sink $i$ is a \emph{bidder} in the auction, seeking to satisfy its 
demand $d_i$. Each source is a \emph{lot} containing the supply $s_j$.
Each unsatisfied bidder $i$ offers a \emph{bid amount}, $b_{ij}$, for some lot 
$j$.
Naturally, the bid $b_{ij}$ must exceed the current price, given by $p_j$.
The bidders each want to minimize their \emph{loss}: the quantity
\begin{equation}
c_{ij} - b_{ij},
\end{equation}
representing the cost-price total for bidder $i$ to obtain lot $j$.
The best possible loss for bidder $i$ is the amount closest to 
zero, as given by the negative scalar
\begin{equation}
\max_{j \in A(i)} \set{c_{ij} - b_{ij}}.
\end{equation}

Each bid can be ``outbid''; that is, superseded by another bidder offering a
higher price.
So long as $b_{ij}$ is greater than the highest previous price for lot $j$, 
designated $p_j$, we know that some bidder (either $i$ or a competitor) 
will claim lot $j$.
Once prices are sufficiently high, lot costs becomes irrelevant, so every lot 
eventually receives a bid.
At that point, the auction ends.

Knowing that we require $b_{ij} > p_j$, the natural question to ask is: how much
larger than $p_j$ should we make $b_{ij}$? As \B{} so eloquently explains 
in~\cite[p.\ 29--30]{Bertsekas1998a}, 
we need to set a minimum \emph{bidding increment}, or \emph{step size}: some
$\varepsilon>0$, such that $b_{ij} \geq p_j + \varepsilon$. Otherwise, the
auction risks stalling if two options are equally optimal.

The assignment problem, as formulated for the original auction method, 
assumes that we have $N$ sources and $N$ sinks, each of which has weight 1, and 
that all costs are integer-valued.
Thus, we can assume that the auction method for the assignment problem operates 
on a special case of our transport problem: one with integer costs and unit 
weights, where $M=N$.
Applying the terminology used by \B{}, we call a sink of weight 1 a 
\emph{person} and a source of weight 1 an \emph{object}.

\subsection{\texorpdfstring{$\varepsilon$}{ε}-complementary slackness}
We can relax the complementary slackness condition, allowing flow from sources
to sinks whenever the loss comes within $\varepsilon$ of attaining the 
maximum.
This is called \emph{$\varepsilon$-complementary slackness}, or 
\emph{$\varepsilon$-CS}, and it can be considered for any transport plan, 
complete or not.

Formally: Given some $\varepsilon > 0$, a simplified transport plan $T$ and 
price vector $p$ satisfy $\varepsilon$-complementary slackness if
\begin{equation}
x_i - \varepsilon
= \max_{k \in A(i)} \set{ c_{ik} - p_k } - \varepsilon
\leq c_{ij} - p_j
\quad
\forall\,
(i,\,j;\,f_{ij}) \in T.
\end{equation}
As \B{} proved in~\cite[p.\ 255--257]{Bertsekas1998a}, the auction method 
maintains $\varepsilon$-CS, and the resulting transport plan is guaranteed to 
be 
optimal if $\varepsilon < 1/N$.

\subsubsection{\texorpdfstring{$\varepsilon$}{ε}-scaling}
\B{} found that the number of iterations of the auction algorithm depends 
on $\varepsilon$ and the number of possible values, or \emph{range}, 
that the cost can take.
When we restrict costs to the negative (or positive) integers, as \B{} does, we 
can safely assume that the cost range $C$ is equal to the maximum absolute cost,
\begin{equation}
C = \max_{(i,\,j) \in \mathcal{A}} \abs{c_{ij}}
= \max_{(i,\,j) \in \mathcal{A}} \set{-c_{ij}}.
\end{equation}
(It is possible for the cost range to be smaller; for example, when 
$\gcd\set{c_{ij}} > 1$.)
As \B{} concluded, for many assignment problems the number of iterations 
required for termination is proportional to $C/\varepsilon$~\cite[p.\ 
34]{Bertsekas1998a}.

Of course, the number of iterations is also dependent on the initial price 
vector; when the initial prices are close to ``$\varepsilon$-optimal,'' the 
number of iterations required is relatively small.
This suggests that it may be advantageous to use a scaling technique, similar 
to that used in penalty and barrier methods.
For the auction algorithm, \B{} calls this technique 
\emph{$\varepsilon$-scaling}.

To perform $\varepsilon$-scaling, we apply the auction algorithm multiple 
times.
Each iteration of the algorithm is called a \emph{scaling phase}.
In the first iteration, we use a simple set of initial prices, along with an 
initial $\varepsilon$.
For each successive iteration, we use the resulting price vector from the 
previous phase, along with an altered $\varepsilon$-value.
The $\varepsilon$-scaling technique terminates when $\varepsilon$ reaches some 
critical value.
When $\varepsilon$-scaling is used, the auction method has worst-case time 
complexity
$\mathcal{O}(NA\log(NC))$, where $A = \abs{\mathcal{A}}$, the number of arcs 
in the underlying transport graph~\cite[p.\ 265]{Bertsekas1998a}.

\subsection{Extended auction}
The extended auction for the integer-valued transport problem was initially 
described by \BC{} in 1989~\cite{Bertsekas1989a}.
They wanted to extend the assignment auction method to one that could handle 
transport problems.
Their idea was to transform the integer-valued transport problem into an 
assignment problem by creating multiple copies:
construct an assignment problem with a number of identical persons (or objects) 
equal to the weight at each sink (or source, respectively), preserving the 
adjacencies and costs of the original vertices.
After solving the assignment problem, redundant arcs with positive flow
could be combined to generate a simplified optimal transport plan.

When a sink or source is split into multiple copies, each with a weight of 1, the resulting 
persons 
and objects retain an identical underlying structure: the adjacencies and cost 
coefficients of the originals.
Such persons and objects are called \emph{similar}.
A \emph{similarity class} is the equivalence class of persons or objects with 
the same adjacencies and costs.
Persons and objects in the same similarity class can engage in 
protracted ``bidding wars'' unless their similarity is taken into account.

\BC{} address the similarity relationships in multiple ways.
Thus, their extended auction method is best understood as three distinct 
algorithms:
\begin{itemize}
\item
The \emph{auction algorithm for the assignment problem}, or simply 
\textbf{AUCTION}, expands the sinks and sources of the transport problem, 
creating an 
assignment problem that it solves without considering the underlying structure.
\item
The \emph{auction algorithm for similar objects}, also called the \textbf{AUCTION--SO}, 
considers the impact of similar objects when increasing prices.
\item
The \emph{auction algorithm for similar objects and persons}, or 
\textbf{AUCTION--SOP}, treats similar persons as a unit during each bidding 
phase.
At every iteration, each unsatisfied similarity class of persons bids 
collectively 
for a number of objects equal to its total demand $d_i$.
Each similarity class of persons shares a single price increase, determined 
similarly to the technique of the AUCTION--SO.
\end{itemize}
For all three algorithms, the costs and weights must be integers. If 
$\varepsilon < 1/\min\set{M,\,N}$, the solution resulting from any of these is 
guaranteed to be optimal~\cite[p.\ 85]{Bertsekas1989a}.

\subsection{Factors motivating the general auction}
The motivations behind proposing the general auction method
are best understood by contrasting it with the 
assignment auction method and its extensions.

All of the auction methods described here are capable of handling real-valued
cost functions without any modification.
In practice, use of real-valued costs (such as the Euclidean distance) does
not generally have a noticeable impact on performance.
However, when one or more
costs is real-valued, the complexity bounds
described in~\cite{Bertsekas1998a} and~\cite{Bertsekas1993a} no longer
apply. Hence, none of the auction methods
offer a guaranteed upper bound on complexity for transport problems
with real-valued costs.

The primary issues arise when one or more weights in $\set{d_i}$ or $\set{s_j}$
are not integer-valued.
The auction methods developed by \BC{} require that the transport problem be transformed
into an assignment problem.
That requires the existence of some common value $V \in \R$ such that $Vd_i$ and $Vs_j$ are
integers for all $i$ and $j$.
Even if such a $V$ exists, the transformation generates an assignment problem with $LV$ sources
and sinks, which must then be solved.
This transformation step can easily result in problems so large as to be unsolvable, as the following
(admittedly degenerate) example illustrates.
\begin{quotation}
Let $\set{a,\,a + 2}$ and $\set{b,\,b + 2}$ be two sets of twin primes.
Choose $M = N = 2$, and
\begin{equation}
d_1 = a,\quad\quad d_2 = b + 2,\quad\quad s_1 = b,\quad\quad s_2 = a + 2.
\end{equation}
Choose any negative-valued cost coefficients $c_{ij}$ for $i = 1,\,2$ and $j = 1,\,2$. The transport
problem is feasible, and its underlying transport graph has four vertices and four arcs.

Now
convert the transport problem to an assignment problem, as required by all of the \BC{} methods.
Because the demand and supply
quantities are relatively prime, it is impossible to ``downsize'' the resulting problem. This
means the underlying assignment graph has $2(a + b + 2)$ vertices and $(a + b + 2)^2$ arcs.
Even if the twin prime conjecture turns out to be false, the largest known twin primes
have more than 200 000 digits (in base-10 notation). Using these twin primes, a transport
problem simple enough to solve by hand becomes an assignment problem requiring more
vertices than the number of electrons in the universe. By comparison, an auction algorithm
that did not require this transformation would require far less storage space. For example,
the general auction would require no more than fourteen weight variables\footnote{Two
demand values $d_i$, two unsatisfied demand values $D_i$, two supply values $s_j$,
two available supply values $S_j$, and up to six quantities $q_{ij}$
(two for bidding and up to four for claim lists); see \cref{s:ga} for details.}.
Given an effective method for storing large integers, they would take up just
over one megabyte of computer memory.
\end{quotation}
If no such $V$ exists, then the weights in $\set{d_i}$ and $\set{s_j}$ are not
rationally-related, and the auction methods of \BC{} cannot be applied at
all\footnote{In theory, it may be possible to successfully approximate the optimal
transport solution by starting from a set of approximated weights. However, to the best of
our knowledge, no one has proposed this approach, or offered any mathematical support for
its validity.}.

In contrast to this bleak picture, the general auction method is designed from the ground up
to handle
any set of real-valued weights without special treatment.
Rather than transform the transport problem into an assignment problem, it operates
directly on the original transport problem, so the number of source and sinks, and
the corresponding difficulty, remain unchanged.
In the following sections, we consider how this is done, and experimentally evaluate the
impact on computational time, storage, and complexity.

\section{General auction for the transport problem}\label{s:ga}
Our general auction method uses the real-valued transport problem as its 
basis, rather than the integer-valued assignment problem.
(Thus, the method's name: it is designed around a more general problem than 
other auction methods.)
To avoid redundancy, we define and explain only those terms and ideas which 
differ from other auction methods.

\subsection{Description and terminology}
The general auction uses a variant form of lot bidding, similar to ``times 
the money'' bidding.
Bidder $i$ has unsatisfied demand $D_i$, so bidder $i$ makes a bid of bid 
amount $b_{ij}$ on lot $j$.
The quantity desired from lot $j$ is set to $q_{ij} = \min\set{D_i,\,s_j}$.
We can write the bid as the pair $(b_{ij},\,q_{ij})$.
The actual bid is understood to be price $b_{ij}$ per $q_{ij}$ items, for a 
total bid value of $b_{ij}q_{ij}$.

Suppose that lot $j$ has available supply $S_j$.
If $q_{ij} \leq S_j$, the desired quantity is immediately available, so bidder 
$i$ 
is awarded a \emph{claim} on lot $j$ of quantity $q_{ij}$ at bid price $b_{ij}$.
This claim, represented by the triple $(i;\, b_{ij},\, q_{ij})$, is added to 
the \emph{claim list} for lot $j$, which we denote by $C_j$.

Such claims can still be outbid.
If $q_{ij} > S_j$, we compare bidder $i$'s offer to those already on the claim 
list.
The difference between $q_{ij}$ and $S_j$ is made up by taking the required 
amount from the lowest priced claim(s) with bid price less than $b_{ij}$.
Only if insufficient low-priced claims exist will bidder $i$ claim less than 
$q_{ij}$.

Even so, as long as we ensure that $b_{ij}$ is greater than the lowest bid 
price 
on lot $j$'s current claim list, we know that bidder $i$ will be able to claim 
some quantity in lot $j$.
To guarantee the occurrence of such a claim, for each lot $j$ we must first 
determine a 
\emph{lot price} $p_j$, defined as
\begin{equation}\label{eq:ga_pj}
p_j := \min\limits_{(i;\,b_{ij},\,q_{ij}) \in C_j} \set{ b_{ij} }.
\end{equation}
(If $C_j$ is empty, let $p_j$ be equal to some initial price $p_j^0$.)
When bidding, we require that bid prices satisfy $b_{ij} \geq p_j + 
\varepsilon$ for some fixed $\varepsilon > 0$.

The lot price vector, $p = \set{p_j}_{j=1}^N$, corresponds to the price vector 
used in the dual profit equation.
Thus, like the assignment auction, the general auction attempts to solve the 
dual problem.
The general auction also uses $\varepsilon$-complementary slackness, which is 
defined identically to the assignment auction.

\subsection{Iteration}
Assume without loss of generality that $M \geq N$; that is, there are at least 
as many sinks as sources.
To initialize the general auction method, one must have a bidding step size 
$\varepsilon > 0$ and an initial lot price vector.
Once initialized, the auction is performed in iterations.
Each iteration consists of two phases: a bidding phase (\cref{sn:bidphase}) and 
a claims phase (\cref{sn:clmphase}).

\subsubsection{Bidding phase of the general auction}
\label{sn:bidphase}
Let $\tilde{I}$ be a nonempty subset of sinks $i$ that are unsatisfied under 
the current
transport plan $T$. For each sink $i \in \tilde{I}$:
\begin{algolist}
\item
Find:
\begin{algolist}
\item
The lot $j_i$ offering best expense, given by
\begin{equation}\label{eq:ga_ji}
j_i := \argmax\limits_{j \in A(i)} \set{ c_{ij} - p_j }.
\end{equation}
\item
The second-best expense, chosen by considering lots other than $j_i$,
\begin{equation}\label{eq:ga_wi}
w_i := \max\limits_{j \in A(i),\, j\neq j_i} \set{ c_{ij} - p_j }.
\end{equation}
If $j_i$ is the only source in $A(i)$, define $w_i$ to be $-\infty$.
(For computational purposes, this can be any value satisfying
$w_i \ll c_{ij_i} - p_{j_i}$.)
\end{algolist}
\item
Compute the bid
$(b_{ij_i},\,q_{ij_i})$,
where the bid price $b_{ij_i}$ is given by
\begin{equation}\label{eq:ga_bij}
b_{ij_i}
:= c_{ij_i} - w_i + \varepsilon
\end{equation}
and the quantity claimed $q_{ij_i}$ is equal to
\begin{equation}
q_{ij_i} := \min \set{ D_i,\,s_{j_i} }.
\end{equation}
\end{algolist}

\subsubsection{Claims phase of the general auction}
\label{sn:clmphase}
For each source $j$, let $I_j$ be the set of sinks from which $j$ received a 
bid in the bidding phase of the iteration.
If $I_j$ is nonempty, for each $i_j \in I_j$ with bid
$(b_{i_jj},\,q_{i_jj})$:
\begin{algolist}
\item
While $S_j < q_{i_jj}$ and $p_j \leq b_{i_jj}$:
\begin{algolist}
\item
Find the lowest-priced claim $c = (k;\,b_{kj},\,q_{kj})$ given by
\begin{equation}
c := \argmin\limits_{(i;\,b_{ij},\,q_{ij}) \in C_j} \set{ b_{ij} }.
\end{equation}
\item\label{al:HCrule}
If $k = i_j$, add $q_{kj}$ to $q_{i_jj}$. 

\quad\quad
[Hungry Cannibal rule; see \cref{sn:hcrule}]
\item
Find the quantity in $c$ to be claimed by bidder $i_j$,
\begin{equation}
q := \min\set{q_{i_jj},\,q_{kj}}.
\end{equation}
\item
Make the quantity $q$ available:
\begin{algolist}
\item
Add $q$ to $S_j$.
\item
Subtract $q$ from $q_{kj}$.
\item
If $q_{kj} = 0$, remove $c$ from $C_j$ and update $p_j$.
\item
Add $q$ to $D_k$.
\end{algolist}
\end{algolist}
\item
Let $q_{i_jj} = \min \set{ q_{i_jj},\,S_j }$, and if $q_{i_jj} > 0$:
\begin{algolist}
\item
Insert $(i_j;\,b_{i_jj},\,q_{i_jj})$ into $C_j$.
\item
Subtract $q_{i_jj}$ from $D_{i_j}$ and $S_j$.
\item
Update $p_j$.
\end{algolist}
\end{algolist}

\subsection{Resulting transport plan and cost}
If all sinks have been satisfied, the general auction terminates. The 
resulting complete transport plan is equal to
\begin{equation}
T = \bigcup_{j=1}^n
\setc{(i,\,j;\,q_{ij})}{ (i;\,b_{ij},\,q_{ij}) \in C_j }
\end{equation}
and the primal cost of $T$ is equal to
\begin{equation}
\sum_{j=1}^n \sum_{(i;\,b_{ij},\,q_{ij})\in C_j} c_{ij}q_{ij}.
\end{equation}
If we want to represent $T$ as a simplified transport plan, there 
is still one more step to perform.
Using the claim lists, we can determine the simplified flow for each $(i,\,j)$ 
as
\begin{equation}
f_{ij} := \sum_{(i;\,b_{ij},\,q_{ij}) \in C_j} q_{ij}.
\end{equation}
(If $C_j$ does not contain a claim by bidder $i$, we assume $f_{ij} = 0$.)
Using these flow values, the simplified complete transport plan $\widetilde{T}$ 
equals
\begin{equation}
\widetilde{T} := \setc{(i,\,j;\,f_{ij})}{(i,\,j) \in \mathcal{A},\,f_{ij} > 0 }.
\end{equation}
The primal cost of the simplified transport plan equals
\begin{equation}
\sum\limits_{(i,\,j) \in \mathcal{A}} c_{ij}f_{ij},
\end{equation}
which is exactly the form used in the transport problem.

\subsection{Hungry cannibals}
\label{sn:hcrule}
The \emph{Hungry Cannibal} (or \emph{HC}) rule is a modification to the process of handling claim 
lists, motivated by a potential slow-down during the iterative process.
In the bidding phase, each sink attempts to maximize the quantity it 
acquires.
However, it is possible for a sink $i$ to bid on a lot $j$ where it already has 
a claim, and for the new claim to supersede some or all of the old one.
When this happens, the old claim is ``cannibalized'' by the new one, and the 
quantity acquired by $i$ is not maximal.
(It may even be zero.)
We call any claim that supersedes an earlier claim by the same bidder
a \emph{cannibal} claim.

It is possible to speed up the bidding process, guaranteeing the maximum 
possible quantity is acquired by each bid, even when confronted by cannibal 
claims.
We can deal with cannibals by implementing the following HC rule:
if a new claim by $i$ cannibalizes a quantity $q$, then
during the claims phase $q$ is added to the quantity desired by the 
new claim.
Thus, $i$ can cannibalize itself during the claims phase, and the ``hunger'' of 
$i$ remains maximal.
As a side effect, the HC rule can unify adjacent bids by the same bidder.

It is worth pointing out that the HC rule is an optimization of the general 
auction.
The rule is not necessary to guarantee termination or bound error, but it 
can significantly speed computation.
The speedup occurs because the HC rule requires little overhead to 
implement, while reducing the number of iterations required for 
convergence.
Numerical tests suggest that the transport plan resulting from the general 
auction will be nearly identical, whether or not the HC rule is used, but that 
time improves measurably when the rule is in place.
The HC rule is labeled as \cref{al:HCrule} of the claims phase of the general 
auction.

For the proofs given below, we assume that the HC rule is in place, ensuring 
that the quantity acquired by every new claim is maximal.
This assumption helps simplify our arguments.

\section{Mathematical Results}
The proofs given in~\cite{Bertsekas1989a,Bertsekas1998a} for the assignment 
auction and its extensions rely heavily on the integral nature of the data.
Our generalization to transport problems with real-valued data requires a 
different approach.

\subsection{Termination of the general auction}
Assuming real-valued data invalidates many of the standard assumptions for 
auction algorithms.
For example, unlike \BC{} in~\cite{Bertsekas1989a}, we cannot assume that any 
quantities claimed are bounded away from zero.
Thus, we must consider the possibility that claimed quantities tend to zero as 
the number of bids goes to infinity.
We also cannot assume that lot price increases are bounded away from 
zero, so
we must consider the possibility that price increases tend to zero as the 
number of bids goes to infinity.
Finally, given the definition of lot prices as a minimum, we cannot assume that 
lot prices change at all.
We must consider the possibility that lot prices remain fixed over 
infinitely many bids.
However, our key result, \cref{th:ga_term}, conclusively shows that none of 
these possibilities can occur; the general auction behaves well when applied to 
real-valued data, and 
terminates after a finite number of iterations.
Because of the above considerations, our approach to proving termination bears 
little resemblance to that used by \BC{} in~\cite{Bertsekas1989a}.

\subsubsection{General auction prices are nondecreasing}
\begin{theorem}
\label{th:ga_pr_nondec}
When applying the general auction method with step size $\varepsilon > 0$, lot 
prices are nondecreasing and each bid price exceeds the current lot price by at 
least $\varepsilon$.
\end{theorem}

\begin{proof}
Assume $\varepsilon > 0$ is given.
Fix the lot $j^*$ and let its lot price before and lot price after iteration be 
given by $p_{j^*}$ and $p_{j^*}'$, respectively.
If no sink bids on $j^*$, then the lot price of $j^*$ does not change, and we 
know that $p_{j^*}'=p_{j^*}$.
Suppose instead that some sink $i$ bids on $j^*$, with bid price $b_{ij^*}$.
From \cref{eq:expn,eq:ga_ji}, we know that the expense of sink $i$'s bid is 
given by 
\begin{equation*}
x_i = \max_{j \in A(i)} \set{c_{ij} - p_j} = c_{ij^*} - p_{j^*}.
\end{equation*}
Thus, the bid price $b_{ij^*}$ is given by
\begin{equation*}
b_{ij^*} = c_{ij^*} - w_i + \varepsilon
= p_{j^*} + x_i - w_i + \varepsilon.
\end{equation*}
By \cref{eq:ga_wi}, $x_i - w_i \geq 0$, so
\begin{equation*}\label{eq:ga_bGTp}
b_{ij^*} = p_{j^*} + x_i - w_i + \varepsilon
\geq p_{j^*} + \varepsilon
> p_{j^*}.
\end{equation*}
Since the new bid $b_{ij^*}$ is at least as high as the current lot price 
$p_{j^*}$, by \cref{eq:ga_pj},
\begin{equation*}
p_{j^*}' \geq \min\set{p_{j^*},\,b_{ij^*}} \geq p_{j^*}.
\end{equation*}
Since this is true for all $i$ that bid on $j^*$, we know $p_{j^*}' \geq 
p_{j^*}$.
Therefore, lot prices are nondecreasing.
\end{proof}

\subsubsection{Steady price implies satisfaction}
\begin{theorem}
\label{th:ga_steadysat}
If a bid by sink $i$ on lot $j$ does not increase the lot price $p_j$, then $i$ 
becomes satisfied. Furthermore, if bidder $k$ was satisfied
by lot $j$ after the last increase of $p_j$, then $k$ remains satisfied after 
$i$'s bid is resolved.
If $i$ should become unsatisfied and $p_j$ has not increased, then $i$ will
bid on a lot of price $p_j$ again.
\end{theorem}

\begin{proof}
Fix lot $j^*$.
Let $p_{j^*}$ be the price of $j^*$ prior to the bid by $i$, and assume the 
second-highest bid price on the claim list $C_{j^*}$ is
\begin{equation*}
\hat{p}_{j^*} > p_{j^*}.
\end{equation*}
If $\abs{C_j} < 2$, it is sufficient to assume $\hat{p}_{j^*} = + \infty$.

Suppose $p_{j^*}'$ is the lot price of $j^*$ at some later time, and that
$p_{j^*}' = p_{j^*}$.

Assume first that there are no satisfied sinks on the claim list of $j^*$, and 
let $i_1$ be the sink currently bidding $(b_{i_1j^*},\,q_{i_1j^*})$ on $j^*$.
Let $q_1$ be equal to
\begin{equation*}
q_1 = \sum_{ \setc{(k;\,b_{kj^*},\,q_{kj^*})}{b_{kj^*} = p_{j^*}}  } q_{kj^*}.
\end{equation*}
Because of the HC rule, we can assume without loss of generality that 
$k \neq i_1$ for all claims of price $p_{j^*}$.

As shown in \cref{th:ga_pr_nondec}, we have $b_{i_1j^*} \geq 
p_{j^*}+\varepsilon$.
Thus, if $q_1 \leq q_{i_1j^*}$, all claims of price $p_{j^*}$ have been overbid 
and the price after $i_1$'s claim satisfies
\begin{equation*}
p_{j^*}' \geq \min \set{b_{i_1j^*},\,\hat{p}_{j^*}} > p_{j^*}.
\end{equation*}
This contradicts our supposition that $p_{j^*}' = p_{j^*}$, and so we 
must have $q_1 > q_{i_1j^*}$.
By definition, $q_{i_1j^*} = \min\set{D_{i_1},\,s_{j^*}}$.
If $q_{i_1j^*} = s_{j^*}$, then we have $q_1 > s_{j^*}$, which contradicts the 
definition of $q_1$ as a quantity claimed on $C_{j^*}$.
Hence, we know $q_{i_1j^*} = D_{i_1}$.
Because $i_1$ claimed the quantity $D_{i_1}$ and the HC rule is in place, it 
must be the case that sink $i_1$ is satisfied and
\begin{equation*}
\hat{p}_{j^*}' = \min\set{b_{i_1j^*},\,\hat{p}_{j^*}} > p_{j^*}.
\end{equation*}

Next, suppose instead that $j^*$ has one satisfied sink, and that the sink was 
satisfied by bidding on $j^*$ without increasing the price $p_{j^*}$.
Without loss of generality, assume the sink is $i_1$.
Suppose some sink $i_2$ bids $(b_{i_2j^*},\,q_{i_2j^*})$ on lot $j^*$.
Because $i_1$ is satisfied, we know $i_2 \neq i_1$.
The remaining quantity available at price $p_{j^*}$ equals
\begin{equation*}
q_2 = \sum_{ \setc{(k;\,b_{kj^*},\,q_{kj^*})}{b_{kj^*} = p_{j^*}}  } q_{kj^*}.
\end{equation*}
Because of the HC rule, we can assume $k \neq i_2$ for all claims of price 
$p_{j^*}$.
(From the steps above, we also know $k \neq i_1$ for all such claims.)

Once again, we have $b_{i_2j^*} \geq p_{j^*}+\varepsilon$. 
Note that the bid price associated with the quantity claimed by $i_1$ exceeds 
$p_{j^*}$, so $i_2$ cannot claim any quantity from $i_1$ unless $q_2 \leq 
q_{i_2j^*}$.
However, in that case all claims of price $p_{j^*}$ have been overbid and the 
price after $i_2$'s claim satisfies
\begin{equation*}
p_{j^*}' \geq \min\set{b_{i_2j^*},\,\hat{p}_{j^*}} > p_{j^*}.
\end{equation*}
This contradicts our supposition that $p_{j^*}' = p_{j^*}$, and so we 
must have $q_2 > q_{i_2j^*}$, and $i_1$ remains satisfied.

By definition, $q_{i_2j^*} = \min\set{D_{i_2},\,s_{j^*}}$.
If $q_{i_2j^*} = s_{j^*}$, once again we have a contradiction because the total 
quantity in $C_{j^*}$ exceeds $s_{j^*}$.
Hence, we know $q_{i_2j^*} = D_{i_2}$.
Because $i_2$ claimed the quantity $D_{i_2}$ and the HC rule is in place, it 
must be the case that sink $i_2$ is also satisfied and
\begin{equation*}
\hat{p}_{j^*}' > p_{j^*}.
\end{equation*}

Continuing inductively, we find that up to $M-1$ distinct bidders can be 
satisfied while maintaining $p_{j^*}' = p_{j^*}$, and that
\begin{equation*}
\hat{p}_{j^*}' > p_{j^*}
\end{equation*}
for each of them.
Because the quantity with price $p_{j^*}$ must be owned by at least one 
bidder, it is not possible for more than $M-1$ distinct bidders to satisfy our 
initial assumption that $p_{j^*}' = p_{j^*}$.

Suppose now that $i_k$ becomes unsatisfied for some $k = 1,\,\ldots,\,M-1$, and 
$p_{j^*}' = p_{j^*}$.
Because expenses are nonincreasing and $p_{j^*}' = p_{j^*}$, $x_{i_kj^*}$ must 
still equal the best expense for bidder $i_k$.
Therefore, during its next bidding phase, $i_k$ will bid on a lot with expense 
equal to $x_{i_kj^*}$.
\end{proof}

\subsubsection{General auction terminates}
\begin{theorem}
\label{th:ga_term}
Given $\varepsilon > 0$, if at least one feasible transport plan exists, then 
the general auction method terminates after finitely many iterations.
\end{theorem}

\begin{proof}
Let $I$ be the set of bidders (sinks) and $J$ be the set of lots (sources) for 
a feasible transport 
problem $\mathcal{T}$.
As a consequence of \cref{th:ga_pr_nondec}, we can partition $J$ into four 
subsets:
\begin{itemize}[leftmargin=2em]
\item[--]
$J^F$, the set of lots which receive finitely many bids,
\item[--]
$J^M$, the set of lots whose prices achieve maximum values, while receiving 
infinitely many bids,
\item[--]
$J^A$, the set of lots whose prices asymptotically approach finite limits, 
without ever achieving those limits, and
\item[--]
$J^\infty$, the set of lots whose prices increase without bound.
\end{itemize}

Now consider four subsets of $I$:
\begin{itemize}[leftmargin=2em]
\item[--]
$\displaystyle{I^F = \set{i\text{ bids finitely many times}}}$,
\item[--]
$\displaystyle{I^M = \set{i\text{ bids on some }j \in J^M 
\text{ infinitely many times}}}$,
\item[--]
$\displaystyle{I^A = \set{i\text{ bids on some }j \in J^A
\text{ infinitely many times}}}$, and
\item[--]
$\displaystyle{I^\infty = \set{i\text{ bids on some }j \in J^\infty
\text{ infinitely many times}}}$.
\end{itemize}
We will show that these four sets partition $I$.

Suppose $i \in I$ such that $A(i) \setminus J^\infty$ is nonempty.
Then there exists $j^* \in A(i) \setminus J^\infty$, and by definition 
$p_{j^*}$ is bounded above.
Thus, after a finite number of iterations, $\forall\, j^\infty \in J^\infty$
\begin{equation*}
x_i = \max_{j \in A(i)} \set{c_{ij} - p_j } \geq c_{ij^*} - p_{j^*}
> c_{ij^\infty} - p_{j^\infty}.
\end{equation*}
Therefore, if $A(i) \setminus J^\infty$ is nonempty, we must have $i \notin 
I^\infty$.
By the contrapositive,
\begin{equation}\label{eq:ga_aisubjinf}
A(i) \subseteq J^\infty,
\quad\quad
\forall\, i \in I^\infty.
\end{equation}
This implies that $I^\infty$ is pairwise disjoint with $I^F$, $I^A$, 
and $I^M$.

Suppose there exists $i \in I \setminus I^F$ such that $J^A \cap A(i) \neq 
\varnothing$ and $J^M \cap A(i) \neq \varnothing$.
By definition, $i$ bids infinitely many times on lots from $J^A$ and/or $J^M$.
Let $p_j^n$ be the lot price of $j$ at the end of the $n$-th iteration.
For all $j \in J^A$, define
\begin{equation*}
l_j := \lim_{n \to \infty} p_j^n.
\end{equation*}
After a finite number of iterations $k$, all lots in $J^F$ no longer 
receive bids and all lots in $J^M$ have reached their fixed prices.
Thus, for all $j \in J^M \cup J^F$, and all iterations $n \geq k$,
\begin{equation*}
 p_j^n = p_j^k.
\end{equation*}
Define
\begin{equation*}
j^M := \argmax_{j \in A(i) \cap J^M}\set{c_{ij} - p_j^k}
\end{equation*}
and
\begin{equation*}
j^A := \argmax_{j \in A(i) \cap J^A}\set{c_{ij} - l_j}.
\end{equation*}
One of two possibilities must exist:
\begin{enumerate}
\item
If $c_{ij^M} - p_{j^M}^k \leq c_{ij^A} - l_{j^A}$, then for any $n \geq k$ we 
have $c_{ij^M} - p_{j^M}^{\hat{k}} < c_{ij^A} - p_{j^A}^{n}$.
Thus, after the $k$-th iteration, $i$ will not bid on lots in $J^M$, and so $i 
\notin I^M$.
\item
If $c_{ij^M} - p_{j^M}^k > c_{ij^A} - l_{j^A}$, then there exists iteration 
$\hat{k} \geq k$ such that $c_{ij^M} - p_{j^M}^k > c_{ij^A} - 
p_{j^A}^{\hat{k}}$.
Because lot prices are nondecreasing, after the $\hat{k}$-th iteration $i$ 
will not bid on lots in $J^A$. Therefore, $i \notin I^A$.
\end{enumerate}
Therefore, $I^M$ and $I^A$ are pairwise disjoint.
By definition, $I^F$ must be pairwise disjoint with both $I^M$ and 
$I^A$, because a sink cannot bid both finitely many and infinitely many times.
Thus, all four sets are pairwise disjoint, and since their union is $I$ 
they constitute a partition.

We will now consider possible elements in the three sets 
$I^M$, $I^A$, and $I^\infty$, to show that these sets are in fact empty.
\begin{enumerate}
\item\label{it:ga_IMempty}
Suppose $I^M$ is nonempty.
After a finite number of iterations all prices in $J^M$ are fixed and all 
bidders in $I^M$ no longer bid on lots outside of $J^M$.
From that iteration on, by \cref{th:ga_steadysat}, each bidder $i \in 
I^M$ must be satisfied after its bid.
Thus, some other bidder(s) must cause each $i \in I^M$ to become unsatisfied 
infinitely many times.
By \cref{th:ga_steadysat} and the definition of $I^M$, each time $i$ becomes 
unsatisfied it will rebid on a lot in $J^M$ with the exact same price as the 
one 
it chose on its previous bid.
Therefore, after a finite number of additional iterations, all $i \in I^M$ will 
only have claims in $J^F$ and $J^M$, and sinks in $I^M$ only become unsatisfied 
as a result of bids from other members of $I^M$ for lots in $J^M$.
As a result, after a finite number of additional iterations the quantity 
claimed from $j \in J^M$ at lot price $p_j$ will only be claimed by bidders in 
$I^M$.

Assume without loss of generality that all of these things have occurred by the start 
of the $k$-th iteration.
Let $D^k$ be the total unsatisfied demand at the start of the $k$-th iteration,
\begin{equation*}
D^k = \sum_{i \in I^M} D_i,
\end{equation*}
and let $Q^k$ be the total supply available at the fixed prices:
\begin{equation*}
Q^k = \sum_{j \in J^M}
\sum_{\substack{(i;\,b_{ij},\,q_{ij}) \in C_j\\b_{ij} = p_j^k}} q_{ij}.
\end{equation*}
Assume without loss of generality that $Q^k = u^kD^k$ for some $u^k$.
Because the prices for lots in $J^M$ remain unchanged, we know $u^k > 1$.
During the round, each unsatisfied bidder in $I^M$ must bid on some lot in 
$J^M$. Because the prices of those lots do not increase, we know that all 
those bidders must be satisfied.
Thus, the total quantity acquired by bidders in $I^M$ during the $k$-th round 
must equal $D^k$.
Since the prices of the lots do not increase, the claimed quantities must have 
been taken from $Q^k$, and so $Q^k$ must have been reduced by $D^k$.
Because the new amounts claimed could come only from previous claims by bidders 
in $I^M$, at the start of the $(k+1)$-st round the unsatisfied demand, 
$D^{k+1}$, must equal 
$D^k$, and the ratio of the available quantity, given by $u^{k+1}$, must 
satisfy $u^{k+1} \leq u^k - 1$.
Therefore, $\ceil{u^k}$ rounds after the $k$-th, the available quantity at the
current prices must have been exhausted, and the price of some lot in $J^M$
must have increased.
This contradicts the definition of $J^M$, and therefore $I^M$ must be empty.
\item\label{it:ga_IAempty}
Suppose $I^A$ is nonempty.
By the definition of $J^A$, for each $j_r \in J^A$, there exists an associated
iteration $k_r$ such that the price at the start of that iteration, 
$p_{j_r}^{k_r}$, satisfies
\begin{equation*}
l_r - p_{j_r}^{k_r} < \varepsilon.
\end{equation*}
Let $k = \max_{r} k_r$.

Assume without loss of generality that by the start 
of iteration $k$ all lots in $J^F$ have also received their final bids.
This implies that all bidders in $I^A$ are now bidding exclusively on lots in 
$J^A$.
Recall from \cref{th:ga_pr_nondec} that each new bid exceeds the current lot 
price 
by at least $\varepsilon$.
Thus, after the $k$-th iteration, for all $i \in I^A$ and all $j \in J^A$, the 
new bid price $b_{ij}$ must exceed $l_j$.

Let $D^k$ be the total unsatisfied demand at the start of the $k$-th iteration,
\begin{equation*}
D^k = \sum_{i \in I^A} D_i,
\end{equation*}
and let $Q^k$ be the total supply available at prices not exceeding the 
asymptotic limits:
\begin{equation*}
Q^k = \sum_{j \in J^A}
\sum_{\substack{(i;\,b_{ij},\,q_{ij}) \in C_j\\b_{ij} \leq l_j}} q_{ij}.
\end{equation*}
Assume without loss of generality that $Q^k = u^kD^k$ for some $u^k > 0$.

Let $i$ be some unsatisfied bidder in $I^A$, bidding on $j \in J^A$.
The bid price offered by $i$, by exceeding $l_j$, exceeds the price of all 
quantities claimed in $Q^k$.
We know from the definition of $J^A$ that the price of lot $j$ does not exceed 
$l_j$.
Thus, all bidders in $I^A$ must be satisfied at the end of their bids.
This implies that the claimed quantities must have been taken from $Q^k$, and 
that
\begin{equation*}
Q^{k+1} = (u^k - 1) D^k.
\end{equation*}
Because no bidders outside $I^A$ will bid again on lots in $J^A$, the 
bidders who become unsatisfied due to the new claims on $Q^k$ must be members 
of $I^A$.
Thus, at the start of the $(k+1)$-st round the unsatisfied demand, $D^{k+1}$, 
must equal 
$D^k$ and the ratio of the available quantity $u^{k+1}$, must satisfy $u^{k+1} 
\leq u^k - 1$.
Therefore, $\ceil{u^k}$ rounds after the $k$-th, the quantity claimed at prices 
below 
asymptotic bounds must have been exhausted, and the price of some lot $j 
\in J^A$ exceeds $l_j$.
This contradicts the definition of $J^A$, and therefore $I^A$ must be empty.
\item\label{it:ga_IINFempty}
Suppose $I^\infty$ is nonempty.
From \cref{it:ga_IMempty,it:ga_IAempty}, we know $I = I^F \cup I^\infty$.
Let $i \in I^\infty$.
\Cref{eq:ga_aisubjinf} implies
\begin{equation*}
\min_{j \in A(i)} p_{j} \geq \min_{j \in J^\infty} p_j,
\end{equation*}
and since $p_j \to +\infty$ for all $j \in J^\infty$, it must be the case that
\begin{equation*}
\min_{j \in A(i)} p_{j} \to +\infty.
\end{equation*}
Therefore,
\begin{equation*}
x_i = \max_{j \in A(i)} \set{c_{ij} - p_j} \to -\infty.
\end{equation*}
Thus, after a finite number of iterations, each lot in $J^\infty$ 
will be satisfied exclusively by the bidders in $I^\infty$.
Otherwise, some bidder in $I^F$ would become unsatisfied infinitely 
many times, causing them to bid infinitely often.
This would contradict the definition of $I^F$.

Furthermore, because $I^\infty$ is nonempty, we know the algorithm must not 
terminate. Thus, after a finite number of iterations there exists at least one 
bidder in $I^\infty$ that is not satisfied, while all bidders in $I^F$ have 
been satisfied by the lots in $J^F$.
It follows that the total demand of the bidders in $I^\infty$ must be
strictly larger than the total supply of the lots in $J^\infty$.
However, by \cref{eq:ga_aisubjinf}, the demand in $I^\infty$ can only be 
satisfied by supply in $J^\infty$.
This contradicts the assumption that a feasible transport plan exists.
Therefore, $I^\infty$ must be empty.
\end{enumerate}
Because $I^M$, $I^A$, and $I^\infty$ are all empty, we know $I = I^F$.
Therefore, the general auction algorithm must terminate after finitely many 
bids.
\end{proof}

\subsection{Optimality of the general auction}
When a real-valued transport problem is solved by computer, some error is 
inevitable, if only from the limits of machine precision.
We show below that the error associated with the general auction method can be 
bounded \emph{a priori} with respect to $\varepsilon$.
As a consequence, the method is convergent with respect to 
$\varepsilon$-scaling.

\subsubsection{Minimum price increase for general auction}
\begin{corollary}
\label{th:ga_pr_inc}
Given a feasible transport problem $\mathcal{T}$, there exists a finite $k \in 
\N$ and 
$\delta > 0$ such that every $k$ bids the price of a lot is guaranteed 
to change, and each time a lot price changes it increases by at least $\delta$.
\end{corollary}

\begin{proof}
This follows from \cref{th:ga_term}, specifically that $J^M \cup J^A = 
\varnothing$.
\end{proof}

\subsubsection{General auction preserves \texorpdfstring{$\varepsilon$}{ε}-CS}
\begin{theorem}
\label{th:ga_eCS}
If a transport plan and lot price vector satisfy $\varepsilon$-CS for the 
general auction method at the start of an iteration, the same is true of the 
transport plan and lot price vector obtained at the end of that iteration.
\end{theorem}

\begin{proof}
Suppose $\varepsilon$-CS holds at the start of the iteration.
Fix lot $j^*$ and let its price before and price after iteration be given by 
$p_{j^*}$ and $p_{j^*}'$, respectively.

Suppose that sink $i$ bids on lot $j^*$ during the iteration, and the lot price 
of $j^*$ changes as a result.
By \cref{eq:ga_ji,eq:ga_bij}
\begin{equation*}
b_{ij^*} = c_{ij^*} - w_i + \varepsilon,
\end{equation*}
which implies, by applying \cref{eq:ga_wi}, that
\begin{equation*}
c_{ij^*}-b_{ij^*} = w_i - \varepsilon
= \max\limits_{j \in A(i),\, j\neq j^*} \set{ c_{ij} - p_j } - \varepsilon.
\end{equation*}
\Cref{eq:ga_pj} guarantees that $p_{j^*}' \leq b_{ij^*}$, so
\begin{equation*}
c_{ij^*}-p_{j^*}' \geq c_{ij^*} - b_{ij^*}
= \max\limits_{j \in A(i),\, j\neq j^*} \set{ c_{ij} - p_j } - \varepsilon.
\end{equation*}
By \cref{th:ga_pr_nondec}, $p_j' \geq p_j$ for all $j$.
Thus,
\begin{equation*}
c_{ij^*}-p_{j^*}'
\geq \max\limits_{j \in A(i), j\neq j^*} \set{ c_{ij} - p_j' } - \varepsilon.
\end{equation*}
Because $c_{ij^*}-p_{j^*}' \geq c_{ij^*}-p_{j^*}' - \varepsilon$,
\begin{equation*}
c_{ij^*}-p_{j^*}'
\geq \max\limits_{j \in A(i)} \set{ c_{ij} - p_j' } - \varepsilon,
\end{equation*}
and $\varepsilon$-CS is satisfied.

Suppose that $i$ has a claim on $j^*$, but the lot price $p_{j^*}$ has not 
changed during the iteration.
Because $\varepsilon$-CS held prior to the iteration and $p_j \leq p_j'$ for 
all $j$,
\begin{align*}
\begin{split}
c_{ij^*} - p_{j^*}' = c_{ij^*} - p_{j^*}
&\geq \max_{j \in A(i)} \set{c_{ij} - p_{j}} - \varepsilon \\
&\geq \max_{j \in A(i)} \set{c_{ij} - p_{j}'} - \varepsilon.
\end{split}
\end{align*}
Thus, $\varepsilon$-CS holds for all claims in every claim list.

Therefore, $\varepsilon$-CS holds for all $(i,\,j;\,q_{ij})$ in the 
transport plan $T$.
\end{proof}

\subsubsection{General auction error bound}
\begin{theorem}
\label{th:ga_Leps}
Let $\varepsilon > 0$ be the step size of the general auction method.
If a feasible transport plan exists, then when the general auction method 
terminates, the resulting feasible transport plan is within $L\varepsilon$ of 
optimal, where
\begin{equation}
L = \sum_{i=1}^M d_i = \sum_{j=1}^N s_j.
\end{equation}
\end{theorem}

\begin{proof}
Assume the transport problem is feasible.
Let $P^*$ be the optimal primal solution to the transport problem,
\begin{equation*}
P^* = \max_{\setc{T \in \mathcal{T}}{T\text{ complete}}}\,\,
\sum_{(i,\,j) \in \mathcal{A}} c_{ij}f_{ij},
\end{equation*}
and $D^*$ be the optimal dual solution
\begin{equation*}
D^* =
\min_{p=\set{p_j}_{j=1}^N}
\set*{ \sum_{i=1}^M d_i\max_{j \in A(i)} \set{c_{ij} - p_j}
+ \sum_{j=1}^N s_jp_j }.
\end{equation*}
Let $T^*$ be the simplified transport plan when the auction 
terminates, and
let
$p^* = (p_1^*,\,\ldots,\,p_N^*)$ be the resulting
price vector.
Let $i$ be any sink.
Suppose $(i,\,j_r;\,f_{ij_r}) \in T^*$, with 
$r \in \set{ 1,\,\ldots,\,t_i }$.
Because $(T^*,\,p^*)$ satisfies $\varepsilon$-CS,
\begin{equation*}
\max_{j \in A(i)} \set{c_{ij} - p_j^*} - \varepsilon
\leq c_{ij_r} - p_{j_r}^*.
\end{equation*}
Therefore, by rearranging and summing terms for all $j_r$,
\begin{align*}
\max_{j \in A(i)} \set{c_{ij} - p_j^*}
+ p_{j_r}^*
&\leq
\varepsilon + c_{ij_r}
\\
f_{ij_r}\left(\max_{j \in A(i)} \set{c_{ij} - p_j^*}
+ p_{j_r}^*\right)
&\leq
f_{ij_r}(\varepsilon + c_{ij_r})
\\
\sum
f_{ij_r}\left(\max_{j \in A(i)} \set{c_{ij} - p_j^*} + p_{j_r}^*\right)
&\leq
\sum
(f_{ij_r}\varepsilon + f_{ij_r}c_{ij_r})
\\
d_i\max_{j \in A(i)} \set{c_{ij} - p_j^*} +
\sum
f_{ij_r}p_{j_r}^*
&\leq
d_i\varepsilon +
\sum
f_{ij_r}c_{ij_r},
\end{align*}
where each summation takes place on all
$(i,\,j_r;\,f_{ij_r}) \in T^*$ over all $r \in \set{ 1,\,\ldots,\,t_i }$.

Summing over all sinks $i$, this gives us
\begin{align}\label{eq:ga_sinksum}
\begin{split}
\sum_{i=1}^M d_i\max_{j \in A(i)} \set{c_{ij} - p_j^*} +
\sum_{(i,\,j;\,f_{ij}) \in T^*} f_{ij}p_j^* \\
\quad\quad\leq
\sum_{i=1}^M d_i\varepsilon +
\sum_{(i,\,j;\,f_{ij}) \in T^*} f_{ij}c_{ij}.
\end{split}
\end{align}
Given any sink $j$, we have $(i_u,\,j;\,f_{i_uj}) \in T^*$ for all
$u \in \set{ 1,\,2,\,\ldots,\,v_j}$.
Thus, summing first over the $i_u$ for $j$, and then over all $j$, we have
\begin{align*}
\sum_{\substack{(i_u,\,j;\,f_{i_uj}) \in T^*\\u \in \set{1,\,\ldots,\,v_j}}}
f_{i_uj}p_j^*
&=
s_jp_j^* \\
\sum_{(i,\,j;\,f_{ij}) \in T^*} f_{ij}p_j^*
&=
\sum_{j=1}^N s_jp_j^*.
\end{align*}
Substituting this into \cref{eq:ga_sinksum}, we have
\begin{align*}
\begin{split}
\sum_{i=1}^M d_i\max_{j \in A(i)} \set{c_{ij} - p_j^*}
+ \sum_{j=1}^N s_jp_j^* \\
\quad\quad\leq
\sum_{i=1}^M d_i\varepsilon +
\sum_{(i,\,j;\,f_{ij}) \in T^*} f_{ij}c_{ij}.
\end{split}
\end{align*}
Therefore,
\begin{align*}
P^*
&= D^*
\\ &\leq
\sum_{i=1}^M d_i\max_{j \in A(i)} \set{c_{ij} - p_j^*}
+ \sum_{j=1}^N s_jp_j^*
\\ &\leq
\sum_{i=1}^M d_i\varepsilon +
\sum_{(i,\,j;\,f_{ij}) \in T^*} f_{ij}c_{ij}
\\ &\leq
L\varepsilon + P^*
\\ &\leq
L\varepsilon + D^*.
\end{align*}
\end{proof}

\subsection{Essential characteristics of the general auction}
In \emph{Network Optimization: Continuous and Discrete Models}, Dimitri 
\B{} describes what he considers the ``important ingredients'' of auction 
methods~\cite{Bertsekas1998a}.
Here, in the same form used by \B{}, are what we consider the essential 
elements of the general auction method:

\bigskip
\begin{center}
\fbox{
  \parbox{.90\textwidth}{
Given $\varepsilon > 0$:
\begin{description}
\item[(a)\phantom{1)}] $\varepsilon$-CS is maintained.
\item[(b-1)] During each iteration, at least one bidder with unsatisfied 
demand 
claims supply in one lot.
\item[(b-2)] The bid price of any claimed supply is increased by at least 
$\beta\varepsilon$, where $\beta$ is some fixed positive constant.
\item[(b-3)] Any previously-claimed supply that is needed to satisfy a 
higher priced claim (if any) becomes unclaimed.
\item[(c-1)] No bid price is decreased.
\item[(c-2)] Any supply that was claimed at the start of an iteration 
remains 
claimed at the end of that iteration (although the bidder claiming it may 
change).
\end{description}
  }
}
\end{center}

\bigskip
With the exception of (a), all of these characteristics are essential to our 
argument that the general auction algorithm terminates after a finite number of 
iterations.
Characteristic (a) allows us to relate the the resulting price vector to the 
optimal price vector we seek, in order to establish a worst-case bound on 
the distance from optimality.

\subsection{AUCTION--SO is a special case of the general auction}
Given an integer-valued transport problem, we can relate the general auction 
method to the extended auction method.
This relationship bypasses the AUCTION--SOP algorithms, and establishes 
the AUCTION--SO as a special case of the general auction.
In \cite{Bertsekas1989a}, \BC{} showed that the original assignment auction 
method is a special case of the AUCTION--SO.
Thus, it follows that the assignment auction developed by \B{} in 
\cite{Bertsekas1981a} is a special case of the general auction.

\begin{theorem}
\label{th:ga_soequiv}
Suppose $\mathcal{T}$ is a feasible integer-valued transport problem and $d_i = 1$ 
for all sinks $i$ in $\mathcal{T}$. Then the general auction algorithm is 
equivalent to the auction algorithm for similar objects.
\end{theorem}

\begin{proof}
Let $\mathcal{T}$ be any feasible integer-valued transport problem such that 
for all sinks $i$, $d_i = 1$.
Thus, for purposes of the AUCTION--SO we have the similarity class $S(i) = 
\set{i}$ for all sinks $i$.
The lot represented by $S(i)$ is unsatisfied if and only if there exists 
exactly one person for the AUCTION--SO that is unsatisfied.
Consider the bidding phase for such a person $i$.

If the object $j_i$ offers best expense for the AUCTION--SO, then the lot 
represented by $S(j_i)$ also offers best expense, so the object $j_i$ chosen by 
the AUCTION--SO corresponds to the choice of lots in the general auction.
Let $j'$ be the object such that for the AUCTION--SO
\begin{equation}
w_i = \max_{j \in A(i) \setminus S(j_i)} \set{c_{ij} - p_j} = c_{ij'} - p_{j'}.
\end{equation}
As a consequence, it must be that
\begin{equation}
p_{j'} = \min_{j \in S(j')} p_j,
\end{equation}
and so the second-best expense computed by the general auction method must 
equal that computed by the AUCTION--SO.
Therefore, the bid prices for the SO and general auction must be equal.
Because $\mathcal{T}$ is an integer-valued transport problem, $S(i)$ is 
unsatisfied, and $d_i = 1$, the quantity desired by the general auction must be
\begin{equation}
q_{ij_i} = \min\set{D_i,\,s_{j_i}} = 1.
\end{equation}
Therefore, the two bidding phases are equivalent.

Let $j$ be some object who receives one or more bids during the claims phase, 
and let $i_j$ be the person that made the highest bid on $j$.

Suppose the object $j$ has already been claimed by some person $k$.
Because $j$ is the lowest-priced object in $S(j)$, and each bid increases the 
price of an object by at least $\varepsilon$, this implies that the lot $S(j)$ 
is unavailable.
The AUCTION--SO ``claim'' on object $j$ corresponds to making a claim on the 
lot 
$S(j)$, and
\begin{equation}
b_{i_jj} = c_{i_jj} - w_{i_j} + \varepsilon
\geq p_j + x_{i_j} - w_{i_j} + \varepsilon
\geq p_j + \varepsilon.
\end{equation}
Thus, in the general auction the lowest priced claim on the current claim 
list corresponding to $S(j)$ will be removed.

Assume without loss of generality that the order of the claims in the claim 
list of the general auction matches the sorted order of the expenses determined 
in the AUCTION--SO algorithm.
Then the lowest priced claim corresponds to the object claimed by person 
$k$.
As shown in the bidding phase, the quantity claimed by $i$ is 1, so both 
methods add 1 to $D_k$.
Removing the claim $(k;\,b_{kj},\,1)$ from the 
claim list corresponding to $S(j)$ is equivalent to removing $(k,\,j;\,1)$ from 
$T$.

Appending $(i_j,\,j;\,1)$ to $T$ is equivalent to inserting the bid 
$(i_j;\,b_{i_jj},\,1)$ into the claim list corresponding to $S(j)$,
and both methods subtract 
1 from $D_{i_j}$.
The price update for the general auction is the same as determining the 
lowest-priced object in $S(j)$ after the price increase on object $j$.
Thus, the two claims phases are equivalent, and so the two auction algorithms 
are 
equivalent.
\end{proof}

\section{Numerical results}\label{sn:numerical}
We created two different versions of the general auction.
The first is integer-based and stores all arc costs; we used it for comparison 
with existing auction methods.
The second uses floating-point numbers and computes arc costs as needed;
we used it to approximate time and memory scaling for the general auction.

We implemented four additional methods for testing our 
numerical results: extended method's AUCTION, AUCTION--SO, and AUCTION--SOP 
algorithms, and the network simplex method.
The assignment and extended auction methods were used for benchmarking and 
comparison, while the network simplex method was used to test real-valued
solutions for optimality.\footnote{The network simplex method is not generally useful
for benchmarking. The algorithm must be initialized by choosing
an initial feasible solution, and the resulting time and storage are extremely
dependent on that choice of initialization. See~\cite{Kovacs2015a} for details on the range
of network simplex performance benchmarks using various libraries and scenarios.}.

All implementations were written in \CC{}, and rely heavily on the \CC{}11 
Standard Library.
We wrote, compiled, and tested all the programs ourselves, in order to minimize 
possible confounds in our results.
As shown below, the AUCTION--SOP algorithm generated the most favorable results 
in our tests on non-assignment problems, so we focused on comparisons involving 
that algorithm.

All of these methods were implemented to solve the general transport problem as 
described in \cref{sn:trans_prob}, applying only those restrictions 
necessary to satisfy the minimal requirements of the algorithm.
One could improve on our results for any particular method, simply by 
customizing the code to handle specific purposes or environments.
However, our goal was to evaluate the average-case effectiveness of the 
underlying methods themselves, aside from any potential time savings due to 
specialized design.

The transport problems we used for comparison testing were initialized by 
creating 
assignment problems with NETGEN~\cite{NETLIB2016a}. The only modifications 
to the NETGEN code were increased array sizes (to generate the large problems 
desired) and alterations to the input/output routines.
The network generation code was not altered.
When different weights or costs were desired, the existing nodes and 
arcs were modified using the random uniform distribution functions from the 
\CC{}11 Standard Library.

In order to obtain accurate comparisons, we used identical conditions for the 
creation and implementation of all programs, and ran all tests in large blocks.
To minimize variation due to underlying problem structure, we solved at least 
10 distinct transport problems for each result. Solution times are averages.
When comparing algorithms, our programs loaded the data and ran multiple 
algorithms in sequence, minimizing changes in conditions over time.
We also ran multiple iterations of each test to ensure that results were 
typical.

Computation times under one second were obtained by forcing multiple runs to
bring the total time over one second, and dividing the total by the number of 
runs to obtain an average.
For example, if a time of $0.9$ seconds is given, that indicates the 
algorithm was run ten times, and the total resulting time divided by ten.
In order to obtain accurate results over multiple runs, we wrote our programs 
to recreate all data structures and computations from scratch.
We also compared our multi-run results to single-run times, to ensure that our 
multi-run averages closely approximated the times given by single runs, which 
they did.

\subsection{Comparison of auction methods for assignment}
When \BC{} compared their implementation of the extended 
auction to the performance of the assignment auction on standard assignment 
problems~\cite[p.\ 92]{Bertsekas1989a}, they found that the additional overhead 
required by the extended auction measurably slowed computation time.
For this reason, we wished to see how the performance of the general auction 
compared to that of the assignment and extended auction methods when applied to 
assignment problems of increasing size.
The comparison done by \BC{} in 1989 used problems with 150 to 500 pairs of 
sources and sinks.
Given the improvements in computation power since that time, our comparison 
starts at 3000 sinks and ends at 10000 (with an equal number of sources at 
each size).
As in~\cite{Bertsekas1989a}, the number of arcs in each problem is 12.5\% of 
the maximum possible.
The cost range is also relevant, as it influences complexity 
scaling~\cite[p.\ 34]{Bertsekas1998a}.
We used the fixed cost range $C = 100$.
The resulting times are given in \cref{tb:as_time}. The SOP auction results are
shown in two columns: one displays times for the original algorithm described
in~\cite{Bertsekas1989a}, while the other is for an optimized version developed
by Dr.\ Walsh and publicly available via~\cite{Walsh2016b}.
Approximating the time complexity of each algorithm with the power regression 
$aN^b$, we find that the original AUCTION--SOP has $b=2.9$, but each of the 
other three algorithms has a nearly identical power: $b = 2.4$.

\begin{table}[htpb]
\centering
\caption{~~Time in seconds for assignment scaling\\$N$ sinks and $N$ sources, 
$N^2/8$ arcs}\label{tb:as_time}%
\setlength\tabcolsep{3pt}
\begin{tabular}{c r c | c r c c r c c r c c r c}
\multicolumn{3}{c |}{} &
\multicolumn{3}{c}{Assignment} &
\multicolumn{3}{c}{General} &
\multicolumn{3}{c}{SOP} &
\multicolumn{3}{c}{SOP} \\
\multicolumn{3}{c |}{\,\,\,$N$} &
\multicolumn{3}{c}{auction} &
\multicolumn{3}{c}{auction} &
\multicolumn{3}{c}{(Walsh)} &
\multicolumn{3}{c}{(original)} \\
\hline
&       & & &       & & &       & & &         & & &        & \\[-2ex]
&  3000 & & &  1.23 & & &  1.21 & & &  1.71 & & &   9.06 & \\
&  4000 & & &  2.65 & & &  2.68 & & &  3.52 & & &  18.63 & \\
&  5000 & & &  4.73 & & &  4.82 & & &  6.23 & & &  35.33 & \\
&  6000 & & &  6.96 & & &  7.06 & & &  8.82 & & &  60.55 & \\
&  7000 & & &  8.90 & & &  8.97 & & & 10.85 & & &  90.82 & \\
&  8000 & & & 12.60 & & & 12.64 & & & 15.98 & & & 132.93 & \\
&  9000 & & & 17.00 & & & 16.83 & & & 20.73 & & & 193.30 & \\
& 10000 & & \phantom{x}
 & 18.88 & &
 & 19.05 & &
 & 23.04 & & \phantom{x}
 & 249.03 &
\end{tabular}
\end{table}

Over all tested problem sizes, the computed times for the assignment and 
general auction methods are nearly identical, differing by less than 2\%.
The largest difference in the times required by these two methods was less 
than two-tenths of one second.
This result suggests that the overhead required for the implementation of claim 
lists 
scales acceptably for relatively simple problems such as assignments.

Storage requirements were dominated by the need to store an explicit cost value 
for each arc, and so were nearly identical for all three algorithms. Storage 
scaled quadratically with respect to the number of sinks.

As these nearly-identical assignment auction and general auction results indicate,
the additional computational overhead required to maintain claim lists is
negligible.
For assignment problems, the performance of the general auction method is
effectively identical to the classical assignment auction.

\subsection{Comparison of extended and general auction}
The experiments described in~\cite{Bertsekas1981a,Bertsekas1989a} rely
on homogeneity and asymmetry for their numerical examples.
We replicated variations of these conditions for our own tests, and considered
what happened when those assumptions were violated.
Due to the significant time delays associated with the original version of the
SOP auction, all SOP tests were performed using the optimized Walsh SOP auction
published in~\cite{Walsh2016b}.

For the assignment auction, we used NETGEN to construct 
systems with $400$ sinks and $N$ sources.
Each of the $N$ sources has unit weight. The $400$ sinks have a 90\%/10\% 
weight 
distribution: 50\% of the weight is uniformly distributed over 90\% of the 
sinks, while the other 50\% is uniformly distributed over the other 10\%.
The density of the arcs is 14\%.
Computation times for all four auction algorithms, with various values of $N$, 
are given 
in \cref{tb:ha_cmp}.
Approximating the time complexity of each algorithm with the power regression 
$aN^b$ gives us the results shown in \cref{tb:ha_prg}.

\begin{table}[htpb]
\centering
\caption{~~Time in seconds for asymmetric problem\\ with unit weights \\ $400$ sinks 
with 90\%/10\% weights, \\$N$ sources with unit weight, 14\% arc 
density}\label{tb:ha_cmp}
\setlength\tabcolsep{3pt}
\begin{tabular}{c r c | c r c c r c c r c c r c}
\multicolumn{3}{c |}{} &
\multicolumn{3}{c}{General} &
\multicolumn{3}{c}{SOP} &
\multicolumn{3}{c}{SO} &
\multicolumn{3}{c}{Assignment} \\
\multicolumn{3}{c |}{$N$} &
\multicolumn{3}{c}{auction} &
\multicolumn{3}{c}{auction} &
\multicolumn{3}{c}{auction} &
\multicolumn{3}{c}{auction} \\
\hline
&       & & &      & & &      & & &       & & &       & \\[-2ex]
&  3000 & & & 0.16 & & & 0.13 & & &  1.45 & & &  1.38 & \\
&  6000 & & & 0.56 & & & 0.32 & & &  6.41 & & &  6.29 & \\
&  9000 & & & 1.06 & & & 0.67 & & & 15.51 & & & 13.30 & \\
& 12000 & & & 1.72 & & & 0.67 & & & 28.37 & & & 27.89 & \\
& 15000 & & & 2.37 & & & 0.99 & & & 40.13 & & & 39.82 & \\
& 18000 & & & 3.17 & & & 1.18 & & & 61.41 & & & 55.42 & \\
& 21000 & & & 3.80 & & & 1.45 & & & 80.09 & & & 86.85 & \\
& 24000 & \phantom{x} &
\phantom{x} & 5.34 & &
& 1.80 & &
& 118.66 & &
\phantom{x} & 114.38 &
\end{tabular}
\end{table}

\begin{table}[htpb]
\centering
\caption{~~Power regression $aN^b$ for asymmetric problem with unit 
weights}\label{tb:ha_prg}
\setlength\tabcolsep{3pt}
\begin{tabular}{c | c l c c l c c r c c r c}
$aN^b$ &
\multicolumn{3}{c}{General} &
\multicolumn{3}{c}{SOP} &
\multicolumn{3}{c}{SO} &
\multicolumn{3}{c}{Assignment} \\
coefficients &
\multicolumn{3}{c}{auction} &
\multicolumn{3}{c}{auction} &
\multicolumn{3}{c}{auction} &
\multicolumn{3}{c}{auction} \\
\hline
    & &      & & &      & & &       & & &       & \\[-2ex]
$a$ & & 0.029 & & & 0.036 & & & 0.15 & & & 0.14 & \\
$b$ & & 1.63  & & & 1.22  & & & 2.07 & & \phantom{xx} & 2.10 &
\end{tabular}
\end{table}

The assignment problem results given in \cref{tb:as_time} illustrate what
happens when we apply the AUCTION--SOP to a homogeneous problem without 
asymmetry.
To evaluate what happens with an asymmetric problem that is not homogeneous, we
implemented the same $400 \times N$ problem as above, but gave the $N$ sources 
weights in 
the integer range $[1,\,19]$.
We used a uniform integer distribution, giving us an average source weight of 
10.
The sinks kept the 90\%/10\% weight distribution.
Even this slight increase in weight resulted in a significant change in the 
relationship between the general auction and AUCTION--SOP, as shown in
\cref{tb:na_cmp}.

\begin{table}[htpb]
\centering
\caption{~~Time in seconds for asymmetric problem \\ with increased weights\\$400$ 
sinks with 90\%/10\% weights, \\ $N$ sources with $[1,\,19]$ weights, 14\% arc 
density}\label{tb:na_cmp}
\begin{tabular}{c r c | c r c c r c}
\multicolumn{3}{c |}{} &
\multicolumn{3}{c}{General} &
\multicolumn{3}{c}{SOP} \\
\multicolumn{3}{c |}{$N$} &
\multicolumn{3}{c}{auction} &
\multicolumn{3}{c}{auction} \\
\hline
&       & & &       & & &      & \\[-2ex]
&  3000 & & &  1.15 & & &  1.88 & \\
&  6000 & & &  2.64 & & &  4.89 & \\
&  9000 & & &  5.40 & & &  8.06 & \\
& 12000 & & &  6.82 & & & 12.18 & \\
& 15000 & & &  8.93 & & & 19.12 & \\
& 18000 & & & 12.08 & & & 25.13 & \\
& 21000 & & & 14.44 & & & 35.58 & \\
& 24000 & \phantom{x} & \phantom{x} & 17.73 & & & 43.81 &
\end{tabular}
\end{table}

Approximating the time complexity of each algorithm with the power regression 
$aN^b$ gives us:
\begin{align*}
a &= 0.27, & b &= 1.31 & \, &\text{ for the general auction, and } \\
a &= 0.32, & b &= 1.51 & \, &\text{ for the AUCTION--SOP. }
\end{align*}

Generally speaking, the transport-to-assignment transformation used by the 
algorithms of the extended auction method degrades performance on transport 
problems with arbitrarily large total weights.
To assess the overall impact, we generated a small transport problem and scaled 
only the total weight, comparing the time required by the extended and general 
auction methods.

In every case, the structure of the transport problem was identical.
It had 500 sources, 500 sinks, and the same 225000 arcs (90\% of maximum).
The cost range was $C=100$.
Using a uniform integer distribution, the weight was spread among the sources 
and sinks.
Each was given a minimum weight of one, but no maximum was assumed.
The total weights and computation times for the two algorithms are given 
in \cref{tb:wt_scale}.
\begin{table}[htpb]
\centering
\caption{~~Results for weight-only scaling\\500 sinks and 500 sources, 
225000 arcs}\label{tb:wt_scale}
\setlength\tabcolsep{3pt}
\begin{tabular}{c r c | c r c c | c r c c}
\multicolumn{3}{c |}{} &
\multicolumn{4}{c |}{Time (sec)} &
\multicolumn{4}{c}{Storage (MB)} \\
\cline{4-11}
     & & &      & &        & &       & &       \\[-2ex]
\multicolumn{3}{c |}{Weight} &
\multicolumn{3}{c}{General} &
Extended &
\multicolumn{3}{c}{General} &
Extended \\
\multicolumn{3}{c |}{$L$} &
\multicolumn{3}{c}{auction} &
auction &
\multicolumn{3}{c}{auction} &
auction \\
\hline
&      & & &       & & & & & & \\[-2ex]
&  500 & & & 0.09 & & \phantom{11}0.11 & & 10.78 & & 10.72 \\
&  600 & & & 0.15 & & \phantom{11}6.35 & & 10.79 & & 11.78 \\
&  700 & & & 0.16 & & \phantom{1}17.72 & & 10.80 & & 12.81 \\
&  800 & & & 0.13 & & \phantom{1}23.16 & & 10.81 & & 13.84 \\
&  900 & & & 0.15 & & \phantom{1}24.85 & & 10.81 & & 14.88 \\
& 1000 & & & 0.15 & & \phantom{1}29.65 & & 10.82 & & 15.91 \\
& 2000 & & & 0.25 & & \phantom{1}51.44 & & 10.88 & & 26.28 \\
& 3000 & & & 0.31 & & \phantom{1}65.29 & & 10.97 & & 36.70 \\
& 4000 & & & 0.40 & & \phantom{1}82.20 & & 11.00 & & 47.02 \\
& 5000 & & & 0.47 & & 102.07 & & 11.08 & & 57.54
\end{tabular}
\end{table}

The data collected suggests that, as weight increases, time for the 
AUCTION--SOP 
algorithm scales on the order of $\mathcal{O}(L^b)$, where $b=1.73$ and
$L$ is the total transport weight.
Time increases for the general auction algorithm as well, but it appears to be 
scaling with $b = 0.63$.
The storage requirements for the AUCTION--SOP scale almost linearly: 
$\mathcal{O}(L^b)$, where $b=0.74$.
This is not surprising, given that the AUCTION--SOP algorithm explicitly 
transforms the set of sinks into $L$ persons.
Storage requirements for the general auction are approximately constant.

The time increases given by both algorithms suggest that the weight range, or 
number of possible weight values, may be relevant to complexity calculations.
The impact of increasing weight was not mentioned in \cite{Bertsekas1989a}, and 
weight changes do not occur in the assignment problem.
However, it seems natural to consider such changes in the context of the 
transport problems.
The larger the weight range, the more likely it becomes that each 
sink will have positive flows over multiple arcs, increasing the complexity of 
the optimal transport plan, and thus the time required to calculate it.

Given the degree to which range increases adversely affect the performance of 
the extended auction method, we also compared the scaling behaviors of the 
AUCTION--SOP and general auction methods in problems where the relative total 
weight was fixed.
We randomly generated transport problems with $N$ sources and $N$ sinks, where 
the total weight of each problem was $L = 2N$.
The number of arcs for each problem was fixed at 90\% of maximum,
and the cost range was $C=100$.
The results are shown in \cref{tb:fx_scale}.
\begin{table}[htpb]
\centering
\caption{~~Results for fixed weight-ratio scaling\\$N$ sinks and $N$ 
sources, $0.9N^2$ arcs, total weight $2N$}\label{tb:fx_scale}
\setlength\tabcolsep{3pt}
\begin{tabular}{r c | c r c c | c r c c}
\multicolumn{2}{c |}{} &
\multicolumn{4}{c |}{Time (sec)} &
\multicolumn{4}{c}{Storage (MB)} \\
\cline{3-10}
     & & &      & &        & &       & &       \\[-2ex]
\multicolumn{2}{c |}{} &
\multicolumn{3}{c}{General} &
Extended &
\multicolumn{3}{c}{General} &
Extended \\
\multicolumn{2}{c |}{$N$\phantom{x}} &
\multicolumn{3}{c}{auction} &
auction &
\multicolumn{3}{c}{auction} &
auction \\
\hline
     & & &      & &        & &       & &       \\[-2ex]
 500 & & & 0.13 & & \phantom{13}8.32 & & 10.82 & & 15.91 \\
 600 & & & 0.19 & & \phantom{1}37.28 & & 20.15 & & 27.53 \\
 700 & & & 0.31 & & \phantom{1}63.27 & & 25.23 & & 35.26 \\
 800 & & & 0.42 & & 177.86 & & 30.77 & & 43.86 \\
 900 & & & 0.61 & & 233.96 & & 36.48 & & 53.05 \\
1000 & & & 0.81 & & 291.29 & & 42.89 & & 63.36 \\
1100 & & \phantom{x} & 1.08 & & 920.57 & & 70.52 & & 95.38
\end{tabular}
\end{table}

The time required by the general auction method scales like $\mathcal{O}(N^b)$ 
with $b = 2.69$, whereas the extended auction scales with $b = 5.40$.
Because of the need to store an explicit cost value for each arc, both methods 
scale quadratically with respect to storage.
At any given size, the extended auction uses approximately 
150\% of the amount of 
storage needed for the general auction method.
The algorithm uses the additional storage to explicitly transform sources into 
objects.

\subsection{General auction performance on real-valued transport}
\label{sn:nm_real}
When evaluating the scaling properties of the real-valued general 
auction method, we wanted to focus on its behavior as problem size increased, 
with all other variables held constant.
To do so, we needed to fix the ranges of both costs and weights.
We constructed transport problems using the following method:
\begin{itemize}
\item
We generate $N$ sinks and an equal number of sources.
Sources and sinks are points in the plane, restricted to the integer 
coordinates on the cube $[0,\,N] \times [0,\,N]$.
Points are placed randomly using the uniform distribution.
\item
The graph underlying our transport problem is
given by
the complete bipartite graph 
$\mathbb{K}_{N,\,N}$, so the problem is maximally dense with $N^2$ arcs.
The cost function is a modification of the Euclidean distance.
Given sink $i = (x_i,\,y_i)$ and source $j = (x_j,\,y_j)$, the cost of the arc 
connecting them is
\begin{multline}
c_{ij} = \sqrt{1 + m_{ij}}, 
\quad\text{ where } \\
m_{ij} = \left[(x_i-x_j)^2 + (y_i-y_j)^2\right] \pmod{20}.
\end{multline}
\item
For $i \in \set{ 1,\,\ldots,\,N}$, the weight of sink $i$ is given by the 
formula
\begin{equation}
d_i = \sqrt{1+m_i}
\quad\text{ where }
m_{i} = [i-1]\pmod{20}.
\end{equation}
and for $j \in \set{ 1,\,\ldots,\,N}$, the weight of source $j$ is given by the 
formula
\begin{equation}
s_j = \sqrt{1+m_j}
\quad\text{ where }
m_{j} = [j-1]\pmod{20}.
\end{equation}
\end{itemize}
We chose this method because it populates the weights and costs with irrational 
values, while restricting the range of those values.
Because the total weight of the sinks equals the total weight of the sources, 
and the transport graph is complete, the problem is guaranteed to be feasible.

The integers are closed under addition and subtraction, so for an integral 
transport problem the smallest deviation from the optimal cost cannot be less 
than one.
Furthermore, the range of possible flow values cannot be greater than the 
largest absolute weight.
Neither of these is true for real-valued problems.
To get a sense of how real-valued data affects these characteristics, consider 
the cost and weight values chosen for our sample problems.
Even without weight splitting, it is possible to make a flow adjustment on the 
smallest possible cycle, four arcs, that changes the cost by
\begin{equation*}
 1 \cdot \left(\sqrt{19} - \sqrt{20} + \sqrt{19} - \sqrt{18} \right)
 \approx 0.00302.
\end{equation*}
From this, one can extrapolate the size of the range.
Increasing the number of arcs in the cycle would further increase the range 
and reduce the minimum possible cost.

As the flow adjustment above indicates, cost and weight ranges are determined,
not by the 20 distinct irrational values given,
but by the number of possible differences that those values can generate.
Theoretically, there could be as many as $\mathcal{O}(V^N)$ differences, where 
$V$ is the number of distinct cost/weight values in the problem. The smallest 
possible cost adjustment could be significantly less than machine precision.

In fact, the cost and weight ranges seemed to have far less negative impact on 
computation than the description above would suggest.
We able to generate optimal solutions in a single iteration of the general 
auction using $\varepsilon = 0.75$.
The results of our tests are given in \cref{tb:ga_optnet}.

\begin{table}[htpb]
\centering
\caption{~~Real-valued general auction results\\$N$ sinks and $N$ sources, 
$N^2$ arcs}\label{tb:ga_optnet}
\begin{tabular}{r c | c r c | c r c}
\multicolumn{2}{c |}{} &
\multicolumn{3}{c |}{Time} &
\multicolumn{3}{c}{Storage} \\
\multicolumn{2}{c |}{$N$} &
\multicolumn{3}{c |}{(sec)} &
\multicolumn{3}{c}{(MB)} \\
\hline
      & & &        & & &       & \\[-2ex]
 1000 & & &   1.11 & & &  3.27 & \\
 2000 & & &   6.37 & & &  5.81 & \\
 3000 & & &  12.74 & & &  7.93 & \\
 4000 & & &  25.87 & & & 14.00 & \\
 5000 & & &  42.49 & & & 14.69 & \\
 6000 & & &  61.14 & & & 15.24 & \\
 7000 & & &  94.12 & & & 16.27 & \\
 8000 & & & 143.51 & & & 18.38 & \\
 9000 & & & 179.20 & & & 25.82 & \\
10000 & & & 218.59 & & & 29.16 &
\end{tabular}
\end{table}

As described above, it can be difficult to determine \emph{a priori} minimum 
$\varepsilon$ values 
that guarantee optimal solutions.
For these tests, we confirmed the optimality of our solutions \emph{a 
posteriori}, 
by comparing the results of the general auction to exact solutions
which were already known.\footnote{That is, the optimal costs, which
were known and unique, could be compared to the cost values we
obtained using the general auction.
Associated transport plans, while known, are not necessarily unique,
and so could not be meaningfully compared to those obtained by our method.}

Given our computer and the type of calculations performed, we have machine 
precision equal to
\begin{equation}
\sqrt{\mathtt{eps}} := 1.490116 \times 10^{-8}.
\end{equation}
The optimal cost values calculated by our general auction
matched the exact solution values up to the limit of machine precision.
Because the optimal costs all exceeded $10^4$, this gave us at least 11 
identical digits.

The time complexity with respect to the number of sinks (or sources) resembles:
\begin{equation*}
a N^b \text{ seconds, }
\quad
\text{ with }
\quad
a = 1.62 \times 10^{-7}
\text{ and }
b = 2.28.
\end{equation*}
When the network simplex method was applied to the same set of
dense, real-valued problems, a time scaling exponent of $b=2.42$ was obtained\footnote{Because exact
network simplex computation times are highly dependent on the choice of initialization, we can make
a meaningful comparison between scaling exponents $b$, but not multipliers $a$.}. 
Because our real-valued implementation of the general auction 
algorithm computed cost values as needed, storage complexity is 
roughly linear.

\section{Conclusions}
\label{sec:conclusions}
We introduced our general auction method for real-valued transport
because we wished to take advantage of the properties of auction algorithms while
overcoming some of the data restrictions imposed by their existing forms.
Our goal was to create an auction technique suitable for computing real-valued 
optimal transport solutions, particularly on asymmetric problems
such as those found in semi-discrete optimal transport.

As we have shown, the general auction allows us to take the same 
rationale as the original auction method and apply it directly to real-valued 
optimal transport problems.
The method is guaranteed to terminate after a finite number of iterations, and 
it offers \emph{a priori} error bounding.

Furthermore, as a discrete transport solver, the general auction method offers 
many practical advantages.
Extending it to parallel computing is relatively straightforward.
By applying \cref{th:ga_Leps}, it is possible to fine-tune the amount of 
computation to a predefined level of acceptable error.
This is a significant departure from most linear programming methods, which 
halt only when the exact solution is reached, and can be arbitrarily far from 
optimal when interrupted.
The results of the general auction include the information necessary 
to compute both transport solutions: the claim lists contain the flows needed 
for the primal solution, and the price vector allows approximation of the dual 
solution.
Together, these two solutions give improved \emph{a posteriori} error 
bounds, which can be used in ``on-the-fly'' $\varepsilon$-scaling schemes.

In the process of developing the general auction for real-valued transport,
we found that very few numerical resources existed for applying auction
algorithms to transport problems of any sort.
In an effort to alleviate this need, we created an open-source software
project that directly implements each of the auction algorithms described
in this paper, allowing them to be compared head-to-head. (A variation of
this software was used to generate our numerical results.)
Our software is publicly available for download at~\cite{Walsh2016b}.

The general auction should also be easy to extend.
In this regard, the efforts of \BC{} provide a road map for the type of 
extensions that may be possible.
In particular, a reverse general auction algorithm seems feasible, as 
does a method for solving partial transport problems.
Given the current interest in fast, accurate techniques for solving optimal 
transport problems, we invite further exploration of the general auction method.

\bibliography{ldjw2016ref}
\end{document}